\documentclass[12pt]{amsart}
\usepackage{mathtools}
\mathtoolsset{showonlyrefs=true}
\usepackage[hmargin=0.8in,height=8.6in]{geometry}
\usepackage{amssymb,amsthm,libertine,libertinust1math}
\usepackage{delarray,verbatim}
\usepackage{ifpdf}
\ifpdf
\usepackage[pdftex]{graphicx}
 \else
\usepackage[dvips]{graphicx}
 \fi
\usepackage{subcaption}

\usepackage{bm}

\usepackage{ifpdf}
\usepackage{color}
\definecolor{webgreen}{rgb}{0,.5,0}
\definecolor{webbrown}{rgb}{.8,0,0}
\definecolor{emphcolor}{rgb}{0.5,0.95,0.95}

\usepackage{hyperref}
\hypersetup{%
	colorlinks=true,
	linkcolor=webbrown,
	filecolor=webbrown,
	citecolor=webgreen,
	breaklinks=true}
\ifpdf \hypersetup{pdftex,
	pdfstartview=FitH, 
	bookmarksopen=true,
	bookmarksnumbered=true
} \else \hypersetup{dvips} \fi
\allowdisplaybreaks

\DeclarePairedDelimiterX{\innp}[1]{\langle}{\rangle}{\ifblank{#1}{\:\cdot\:}{#1}}
\DeclarePairedDelimiterX{\floor}[1]{\lfloor}{\rfloor}{\ifblank{#1}{\:\cdot\:}{#1}}
\DeclarePairedDelimiterX{\ceil}[1]{\lceil}{\rceil}{\ifblank{#1}{\:\cdot\:}{#1}}
\let\abs\relax\DeclarePairedDelimiterX{\abs}[1]{\lvert}{\rvert}{\ifblank{#1}{\:\cdot\:}{#1}}
\let\norm\relax\DeclarePairedDelimiterX{\norm}[1]{\lVert}{\rVert}{\ifblank{#1}{\:\cdot\:}{#1}}

\numberwithin{equation}{section}

\newtheorem{theorem}{Theorem}[section]
\newtheorem{proposition}{Proposition}[section]

\newtheorem{remark}{Remark}[section]

\numberwithin{remark}{section} 
\numberwithin{proposition}{section}
\numberwithin{corollary}{section}

\begin{document}
	
\title[The frequency process in a non-neutral two-type CBSP with competition and its genealogy]{The frequency process in a non-neutral two-type continuous-state branching process with competition and its genealogy}

\author[I. Nu\~nez]{Imanol Nu\~nez$^*$}
\thanks{$*$\, Department of Probability and Statistics, Centro de Investigaci\'on en Matem\'aticas A.C. Calle Jalisco s/n. C.P. 36240, Guanajuato, Mexico. Email: imanol.nunez@cimat.mx.  }
\author[J. L. P\'erez]{Jos\'e-Luis P\'erez$^\dagger$}
\thanks{$\dagger$\, Department of Probability and Statistics, Centro de Investigaci\'on en Matem\'aticas A.C. Calle Jalisco s/n. C.P. 36240, Guanajuato, Mexico. Email: jluis.garmendia@cimat.mx.  } 

\begin{abstract}
We consider a population growth model given by a two-type continuous-state branching process with immigration and competition, introduced by Ma in \cite{maStochasticEquationsTwotype2014}. 
We study the relative frequency of one of the types in the population when the total mass is forced to be constant at a dense set of times.
The resulting process is described as the solution to an SDE, which we call \textit{the culled frequency process}, generalizing the $\Lambda$-asymmetric frequency process introduced by Caballero et al. in \cite{caballeroRelativeFrequencyTwo2024}. 
We obtain conditions for the culled frequency process to have a moment dual and show that it is given by a branching-coalescing continuous-time Markov chain that describes the genealogy of the two-type CBI with competition. 
Finally, we obtain a large population limit of the culled frequency process, resulting in a deterministic ordinary differential equation (ODE).
Two particular cases of the limiting ODE are studied to determine if general two-type branching mechanisms and general Malthusians can lead to the coexistence of the two types in the population. 
\end{abstract}

\maketitle


\section{Introduction}
\label{sec:intro}

The main result by Birkner et al. in \cite{birknerAlphaStableBranchingBetaCoalescents2005} shows that the frequency process describing the proportion of the mass descended from an ancestral individual associated with a measure-valued branching process can be transformed into a $\Lambda$-Fleming--Viot process using a suitable time change involving a functional of the total mass if and only if the measure-valued process is $\alpha$-stable.  
In particular, this implies that if $X^{(1)}$ and $X^{(2)}$ are independent and identically distributed $\alpha$-stable continuous-state branching processes (CSBPs), then the ratio process $R = X^{(1)} / (X^{(1)} + X^{(2)})$ can be time-changed to become Markovian. 
This allows describing the time-changed genealogy of an $\alpha$-stable CSBP through the block-counting process of a $\mathrm{Beta}(2 - \alpha, \alpha)$-coalescent 
that is in moment duality to the aforementioned Markovian version of the frequency process.
Foucart and Hénard, in \cite{foucartStableContinuousstateBranching2013}, extended this result by adding an immigration term to the CSBP, making it a continuous-state branching process with immigration (CBI). 
This modification is reflected in the genealogy, which is now described by the block-counting process of a $\mathrm{Beta}(2 - \alpha, \alpha - 1)$-coalescent, where the difference with the result of Birkner et al. is due to a mutation mechanism in the genealogy that arises because of immigration in the CBI.   
Nonetheless, their work remains in the independent and identically distributed $\alpha$-stable case as they use the time change methodology of Birkner et al.

In \cite{caballeroRelativeFrequencyTwo2024}, Caballero et al. studied the case where $X^{(1)}$ and $X^{(2)}$ are modeled by CBIs with different reproduction mechanisms that are not necessarily $\alpha$-stable, enabling the study of a population with two species reproducing by different mechanisms. 
In their work, the processes are still assumed to be independent. 
The time change perspective does not apply here, see Lemma 3.5 of \cite{birknerAlphaStableBranchingBetaCoalescents2005}. 
Instead, inspired by the works \cite{gillespieNalturalSelectionWithingeneration1974} and \cite{gillespieNaturalSelectionWithingeneration1975} of Gillespie, Caballero et al. used a novel method called culling (see section 4.2 \cite{caballeroRelativeFrequencyTwo2024}) that allows them to control the total population size, see Section \ref{subsec:culling}.
By using this technique, they were able to extend the result proved by Birkner et al. in \cite{birknerAlphaStableBranchingBetaCoalescents2005} to non-$\alpha$-stable cases with different reproduction mechanisms. 
The resulting frequency process was called the \emph{$\Lambda$-asymmetric frequency process}.  

In certain cases, the $\Lambda$-asymmetric frequency process has a moment dual that is the block-counting process of a branching-coalescing process, which is reminiscent of the ancestral selection graph of Krone and Neuhauser \cite{kroneAncestralProcessesSelection1997}. 
The resulting genealogy includes mutation and pairwise branching, reflecting, respectively, immigration in the associated two-type CBI and efficiency in resource consumption (as in \cite{GONZALEZCASANOVA202033}). 
  
Gónzalez-Casanova et al. in \cite{casanovaAlphastableBranchingBetafrequency2024} dropped the independent and identically distributed assumptions made by Birkner et al. on $X^{(1)}$ and $X^{(2)}$ by considering a two-type CBI $X=(X^{(1)}, X^{(2)})$. 
Gónzalez-Casanova et al. showed that the frequency process $R$ can still be time-changed to be Markovian by keeping only the $\alpha$-stability assumption on the associated two-type CBI. 
In some cases, the time-changed frequency process coincides with the $\beta$-Fleming--Viot process defined by Griffiths in \cite{griffithsMultiTypeLCoalescent2016}, which is the weighted moment 
dual to the block-counting process of a two-type branching-coalescing process.  
This connection is enlightening as it implies that the time-changed genealogy of the two-type $\alpha$-stable CBI $X$ is described by a branching-coalescing continuous-time Markov chain. 

This paper aims to generalize these results to a population with two types of individuals ($1$ and $2$), where: i) individuals of one type might produce individuals of the other type (cross-branching); ii) the reproduction mechanisms of each one of the two types might be different; iii) there might be limited resources for which both types compete, leading to different carrying capacities according to the efficiency of the types in resource usage. 
The two types of individuals will be modeled by $X^{(1)}$ and $X^{(2)}$, the components of a two-type population growth model $X$ that extends a CBI using general Malthusians (inspired by Lambert's branching process with logistic growth \cite{lambertBranchingProcessLogistic2005}), introduced by Ma in \cite{maStochasticEquationsTwotype2014}, which we call a two-type CBI with competition. 
Hence, in this paper, we will not assume that $X^{(1)}$ and $X^{(2)}$ are independent, identically distributed nor $\alpha$-stable. 
Our approach consists of applying the culling technique developed in \cite{caballeroRelativeFrequencyTwo2024} to obtain the frequency process of individuals of type $1$ in the population, under the assumption that the total population size remains constant in a dense set of times. We call the resulting process the \textit{culled frequency process}. 
This will allow us to study the dynamics of the genetic profile of populations with two types of individuals where the two types are closely related by the reproduction mechanisms and the competition for limited resources. 

The model introduced in this paper can be valuable for studying populations with two types of individuals that interact and compete for limited resources (see \cite{sinervoRockPaperScissors1996} for an example in ecology).
By considering competition and resource constraints, our model provides insights into how genetic diversity is maintained or altered, and how certain traits prevail or coexist under selective pressures. This framework is a powerful tool for exploring genetic variation, reproduction, mutation, and competition in evolutionary biology, ecology, and conservation biology.

The outline of the paper is the following. 
In section \ref{sec:preliminaries}, we establish the basic notation used throughout the paper and also introduce the population growth model $X$ studied in this work. 
In section \ref{sec:ratio_process} we define the culled frequency process of individuals of type 1, characterizing its dynamics 
through a stochastic differential equation (SDE), see \eqref{eq:sde_r_general}, that is seen to be the limit of the culling procedure, given in section \ref{subsec:culling}. 
The main result in section \ref{sec:ratio_process} is Proposition \ref{prop:important_uniqueness_existence}, which ensures the existence of a unique strong solution to the SDE \eqref{eq:sde_r_general}. 
The proof relies on a simple but enlightening extension of Theorem 3.2 in \cite{fuStochasticEquationsNonnegative2010}. 
With the culled frequency process defined, in section \ref{sec:duality}, we establish a result stating that whenever it has a moment dual, it is a branching-coalescing continuous-time Markov chain, 
describing the genealogy of the two-type CBI with competition. 
Finally, in section \ref{sec:large_population}, the large population limit of the frequency process is considered, resulting in a deterministic ordinary differential equation (ODE), see \eqref{eq:limit_r_ode}.  
Two particular cases of the limiting ODE are studied to determine if general two-type branching mechanisms and general Malthusians can lead to the coexistence of the two types in the population.  

\section{Notation and preliminaries}
\label{sec:preliminaries}

In this section, we introduce the notation we will use in the rest of the paper and the two-type CBI with competition that we will consider for our results. 
In what follows we take $c \in \mathbb{R}_+^2$ and $\eta \in \mathbb{R}_+^2$ to be 
vectors of non-negative constants. 
Until further notice, for $i, j \in \{1, 2\}$ let $b_{ij} : \mathbb{R} \to \mathbb{R}$ 
be a locally Lipschitz function such that 
$b_{ij}(x) = 0$ for $x \leq 0$ and $b_{ij}(x) \geq 0$ for $i \neq j$. 
Moreover, we consider Borel measures $\mu_1, \mu_2$ and $\nu$ over 
$U_2 := \mathbb{R}_+^2 \setminus \{0\}$ that fulfill the integrability conditions 
\begin{align*}
    & \int_{U_2} (1 \wedge \norm{w}) \nu(dw) < \infty, \\
    & \int_{U_2} \bigl( \norm{w} \wedge \norm{w}^2 + w_j \bigr) \mu_i(dw) 
    < \infty, \quad i \neq j, 
\end{align*}
where $\norm{w} = (w_1^2 + w_2^2)^{1/2}$ for $w \in U_2$. 

Extending one of the special cases in \cite{liStrongSolutionsJumptype2012}, see also \cite{maStochasticEquationsTwotype2014}, we 
consider $X = (X^{(1)}, X^{(2)})$ as the solution of the stochastic differential equation 
defined by 
\begin{equation} \label{eq:branching_sde}
\begin{split}
    X_t^{(i)}
    = & x^{(i)} 
    + \int_0^t \bigl( \eta_i + b_{ii}(X_s^{(i)}) + b_{ij}(X_s^{(j)}) 
    \bigr) ds 
    + \int_0^t \sqrt{2 c^i X_{s-}^{(i)}} dB_s^{(i)} \\
    & + \int_0^t \int_{U_2} \int_0^{X_{s-}^{(i)}} w_i \widetilde{N}_R^i(ds, dw, du) 
    + \int_0^t \int_{U_2} \int_0^{X_{s-}^{(j)}} w_i N_R^{j}(ds, dw, du) \\
    & + \int_0^t \int_{U_2} w_i N_I(ds, dw) ,
\end{split}
\end{equation}
for $i, j \in \{1, 2\}$, $i \neq j$, and all $t \geq 0$ such that the SDE makes sense (before the time of explosion). 
In~\eqref{eq:branching_sde} we consider:
\begin{itemize}
    \item $x^{(i)} \geq 0$, for $i \in \{1, 2\}$, is deterministic; 
    \item $B = (B^1, B^2)$ is a two-dimensional Brownian Motion; 
    \item $N_I(ds, dw)$ is a Poisson random measure over $\mathbb{R}_+ \times U_2$ with 
        intensity measure $ds \nu(dw)$; 
    \item for $i \in \{1, 2\}$, $N_R^i(ds, dw, du)$ is a Poisson random measure over 
        $\mathbb{R}_+ \times U_2 \times \mathbb{R}_+$ with intensity measure 
        $ds \mu_i(dw) du$, and $\widetilde{N}_R^{(i)}$ denotes the associated 
        compensated random measure.
\end{itemize}
All of the previous elements are assumed to be defined in the same complete probability space and independent. 
Note that the integrability conditions for $\mu_1$ and $\mu_2$ tell us that these two measures have finite means. 
This allows us to define $X$ through the SDE \eqref{eq:branching_sde}.

The fact that there exists a unique strong solution up to the time of explosion, in case it is finite, to \eqref{eq:branching_sde} is 
due to Theorems 3.1 and 3.2 in \cite{maStochasticEquationsTwotype2014}. 
Indeed, the locally Lipschitz condition allows us to obtain pathwise uniqueness up to the time of explosion by virtually the same argument that proves Theorem 3.1 in 
\cite{maStochasticEquationsTwotype2014}. 
Meanwhile, the existence of a unique positive strong solution follows by a slight 
modification of the proof of Theorem 3.2 in \cite{maStochasticEquationsTwotype2014}.

For the rest of the paper, we will set $z = x^{(1)} + x^{(2)}$ and 
$r = x^{(1)} / z$.  
Moreover, given $z > 0$ we define the function $T_{z, 2} : \mathbb{R}_+^2 \to [0, 1]^2$ 
by 
\[
    T_{z, 2}(w) = \biggl( \frac{w_1}{z + w_1 + w_2}, \frac{w_2}{z + w_1 + w_2} \biggr), 
\]
and the operator  
$\mathbf{T}_2^{(z)} : \mathcal{M}(\mathbb{R}_+^2) \to \mathcal{M}([0, 1]^2)$ 
that maps every measure $\mu$ over $\mathbb{R}_+^2$ to the measure 
$\mu \circ T_{z, 2}^{-1}$ over $[0, 1]^2$; i.e. 
\begin{equation} \label{eq:t2Def}
    \mathbf{T}_2^{(z)} \mu (A) \coloneq \mu \circ T_{z, 2}^{-1}(A) 
    \quad\text{for all } A \in \mathcal{B}([0, 1]^2).
\end{equation}

\section{The culled frequency process}
\label{sec:ratio_process}
Our motivation is to study the dynamics of the genetic profile of a population consisting of two types of individuals who might reproduce using different mechanisms that in general are not independent, subject to competition and immigration.
To this end, we start this section by introducing two processes that characterize the relative frequency of individuals of type $1$ and the total size of the population.

\subsection{Total population and frequency process}
\label{subsec:totalPopRatioProcess}
We start by defining the total population process $Z$ by 
\[
Z_t = X_t^{(1)} + X_t^{(2)}, \quad t \geq 0,\ Z_0 = z,
\]
where we assume that $z > 0$, recalling that $z = x^{(1)} + x^{(2)}$. 
Additionally, we introduce the frequency process of type $1$ individuals given by
\[
    R_t = \frac{X_t^{(1)}}{X_t^{(1)} + X_t^{(2)}} 1_{\{t < \zeta\}} + \dagger 1_{\{t \geq \zeta\}} , \quad t \geq 0, \ R_0=r,
\]
where $\dagger$ is a cemetery state, $\zeta := \inf\{t \geq 0 : Z_t \in \{0,\infty\}\}$ and $r = x^{(1)} / (x^{(1)} + x^{(2)})$. 
We introduce the cemetery state in order to define the process $R$ when the process $Z$ explodes or reaches zero in finite time. 
We note that the frequency process $R$ is not Markovian. 
To address this issue, throughout this section, we will use the culling method developed in \cite{caballeroRelativeFrequencyTwo2024} to construct a Markov process that encodes the dynamics of the frequency process $R$.

It is readily seen that $(R, Z)$ is a Markov process with values in 
$([0, 1] \cup \{\dagger\}) \times [0, + \infty]$. To avoid using the cemetery state, given $0 < \varepsilon < z < L$, we will consider the 
stopping time $\tau := \tau_{\varepsilon}^- \wedge \tau_L^+$, where 
\[
    \tau_\varepsilon^- := \inf\{ t \geq 0 : Z_t = \varepsilon \}
    \quad\text{and}\quad 
    \tau_L^+ := \inf\{t \geq 0 : Z_t > L\}.
\]
Then the process $(R, Z)$ stopped at time $\tau$ will encode the dynamics of the 
populations described by $X$ before the total population $Z$ becomes relatively small 
or explodes. 
This is formalized by a local martingale problem that will be stated in the 
next result, which is a simple consequence of the SDE \eqref{eq:branching_sde} that 
defines $X$ and Itô's formula.

\begin{proposition} \label{prop:rz_generator}
    For every $f \in C^2([0, 1] \times \mathbb{R}_+)$, the process 
    \[
        M_t := f(R_{t \wedge \tau}, Z_{t \wedge \tau}) - f(r, z) 
        - \int_0^{t \wedge \tau} \mathcal{L} f(R_s, Z_s) ds
    \]
    is a local martingale for 
    \begin{align}
        \mathcal{L} f(r, z) = {}
        & \bigl[ \eta_1 + b_{11}(rz) + b_{12}((1 - r) z) \bigr] 
        \biggl[ \partial_1 f(r, z) \frac{1 - r}{z} + \partial_2 f(r, z) \biggr] \\
        & + \bigl[ \eta_2 + b_{21}(rz) + b_{22}((1 - r) z) \bigr] 
        \biggl[ - \partial_1 f(r, z) \frac{r}{z} + \partial_2 f(r, z) \biggr] 
        \label{eq:rz_generator} \\ 
        & + \frac{c_2 - c_1}{z} \partial_1 f(r, z) 2 r (1 - r) 
        + \frac{r (1 - r)}{z} \bigl[ (1-r) c_1 + r c_2 \bigr] \partial_{11}^2 f(r, z) \\
        & + c_1 \bigl[ r (1 - r) \partial_{12}^2 f(r, z) 
        + r (1 - r) \partial_{21}^2 f(r, z) + r z \partial_{22}^2 f(r, z) \bigr] \\
        & + c_2 \bigl[ - r (1 - r) \partial_{12}^2 f(r, z) 
        - r (1 - r) \partial_{21}^2 f(r, z) + (1 - r) z \partial_{22}^2 f(r, z) \bigr] \\
        & + r z \int_{U_2} \biggl[ f \biggl( r + (1 - r) \frac{w_1}{z + w_1 + w_2} - 
        r \frac{w_2}{z + w_1 + w_2}, z + w_1 + w_2 \biggr) - f(r, z) \\ 
        & \hspace*{7cm} - w_1 \biggl( \frac{1 - r}{z} \partial_1 f(r, z) 
        + \partial_2 f(r, z) \biggr) \biggr] \mu_1(dw) \\ 
        & + (1 - r) z \int_{U_2} \biggl[ f \biggl( r + (1 - r) \frac{w_1}{z + w_1 + w_2} -
        r \frac{w_2}{z + w_1 + w_2}, z + w_1 + w_2 \biggr) - f(r, z) \\ 
        & \hspace*{7cm} - w_2 \biggl( - \frac{r}{z} \partial_1 f(r, z) 
        + \partial_2 f(r, z) \biggr) \biggr] \mu_2(dw) \\ 
        & + \int_{U_2} \biggl[  
                f \biggl( r + (1 - r) \frac{w_1}{z + w_1 + w_2} - {}
            r \frac{w_2}{z + w_1 + w_2}, z + w_1 + w_2 \biggr) - f(r, z) \biggr] \nu(dw) 
    \end{align}
\end{proposition}

It is important to remark that even if the process $(R, Z)$ is Markovian, in general 
$R$ will not be an autonomous process (i.e. it will depend on $Z$), hence it is not 
Markovian. 

\subsection{Culled frequency process} 
\label{subsec:asymmetric}
Heuristically, if we want a Markov process that captures the dynamics of the frequency 
process $R$, it should have an infinitesimal generator given by \eqref{eq:rz_generator} 
with a caveat: the function $f$ should only depend on $r$. 
Thus, the desired process should not depend on $Z$, which would correspond to $z$ being 
constant in the generator. 
The culling procedure introduced in \cite{caballeroRelativeFrequencyTwo2024} allows to 
achieve this as a limiting procedure by keeping $Z$ equal to a constant $z>0$ in a dense set 
of times over $\mathbb{R}_+$.  
We now introduce the limiting process of the culling procedure, $R^{(z,r)}$, with constant population size $z>0$ and started at $r$, which we simply call 
\emph{the culled frequency process}, which is a generalization of the $\Lambda$-asymmetric frequency process introduced by Caballero et al. in \cite{caballeroRelativeFrequencyTwo2024}.

For $z > 0$ and $r \in [0, 1]$, let $R^{(z, r)}$ be the solution to the stochastic 
differential equation 
\begin{align}
    R_t^{(z, r)} 
    = {} 
    & r + \int_0^t \bigl[ \widetilde{D}(R_{s-}^{(z, r)}) + S(R_{s-}^{(z, r)})  
    + S_c(R_{s-}^{(z, r)}) +  m(R_{s-}^{(z, r)}) + m_c^1(R_{s-}^{(z, r)}) 
    + m_c^2(R_{s-}^{(z, r)}) \bigr] ds \\ 
    & + \int_0^t \sigma(R_{s-}^{(z, r)}) dB_s 
    + \int_0^t \int_{U_2} \bigl[ \widetilde{g}^{(z)}(R_{s-}^{(z, r)}, w) + 
    \widetilde{h}^{(z)}(R_{s-}^{(z, r)}, w) \bigr] N_3(ds, dw) 
    \label{eq:sde_r_general} \\
    & + \int_0^t \int_{U_2} \int_0^\infty \bigl[ g_1^{(z)}(R_{s-}^{(z,r)}, w, v) + 
    h_1^{(z)}(R_{s-}^{(z,r)}, w, v) \bigr] \widetilde{N}_{1}(ds, dw, dv)  \\
    & + \int_0^t \int_{U_2} \int_0^\infty \bigl[g_2^{(z)}(R_{s-}^{(z,r)}, w, v)
    + h_2^{(z)}(R_{s-}^{(z,r)}, w, v) \bigr] \widetilde{N}_{2}(ds, dw, dv)  
\end{align}
where $r > 0$, 
\begin{equation} \label{eq:sdeTerms}
\begin{split}
    \widetilde{D}(r) = {} 
    & \biggl[ b_{11}(zr) \frac{1 - r}{z} - b_{22}(z (1 - r)) \frac{r}{z} 
    + b_{12}(z (1 - r)) \frac{1 - r}{z} - b_{21}(zr) \frac{r}{z} \biggr] 
    1_{\{r \in [0, 1]\}} , 
    \\
    S(r) = {} 
    & 
    \frac{2}{z} (c_2 - c_1) 
    r (1 - r) 1_{\{r \in [0, 1]\}}, 
    \\
    S_c(r) = {}
    & \biggl[ \sum_{i = 1}^2 (-1)^i \int_{U_2} 
    \frac{w_i (w_1 + w_2)}{z + w_1 + w_2} \mu_i(dw) \biggr] r (1 - r) 
    1_{\{r \in [0, 1]\}} ,
    \\
    m(r) = {} 
    & \frac{1}{z} \bigl[ \eta_1 (1 - r) - \eta_2 r \bigr], 
    \\
    m_c^1(r) = {} 
    & - r^2 \int_{U_2} \frac{z w_2}{z + w_1 + w_2} \mu_1(dw) 1_{\{r \in [0, 1]\}} ,
    \\
    m_c^2(r) = {} 
    & (1 - r)^2 \int_{U_2} \frac{z w_1}{z + w_1 + w_2} \mu_2(dw) 1_{\{r \in [0, 1]\}} ,
    \\
    \sigma(r) = {}
    & \sqrt{ \frac{2}{z} r (1 - r) \bigl[ c_1 (1 - r) + c_2 r \bigr] } 
    1_{\{r \in [0, 1]\}},  
\end{split}
\end{equation}
and 
\begin{itemize}
    \item $B$ is a Brownian Motion; 
    \item for $i \in \{1, 2\}$, $N_i(ds, dw, dv)$, is a Poisson random measure over 
        $\mathbb{R}_+ \times U_2 \times \mathbb{R}_+$ with intensity measure 
        $ds \mu_i(dw) dv$, and $\widetilde{N}_i$ denotes the associated compensated 
        random measure; 
    \item $N_3(ds, dw)$ is a Poisson random measure over $\mathbb{R}_+ \times U_2$ with 
        intensity measure $ds \nu(dw)$; 
    \item for $q, w \in \mathbb{R}_+ \times U_2$ we set 
        \[
            \widetilde{g}^{(z)}(q, w) 
            := (1 - q) \frac{w_1}{z + w_1 + w_2} 1_{\{q \in [0, 1]\}} 
            \quad\text{and}\quad 
            \widetilde{h}^{(z)}(q, w) 
            := - q \frac{w_2}{z + w_1 + w_2} 1_{\{q \in [0, 1]\}} ;
        \]
    \item for $q, w, v \in \mathbb{R}_+ \times U_2 \times \mathbb{R}_+$ we define 
        \begin{align*}
            g_1^{(z)}(q, w, v) & := \widetilde{g}^{(z)}(q, w) 1_{\{v \leq qz\}} , \\
            g_2^{(z)}(q, w, v) & := \widetilde{g}^{(z)}(q, w) 1_{\{v \leq (1 - q)z\}} \\
            h_1^{(z)}(q, w, v) & := \widetilde{h}^{(z)}(q, w) 1_{\{v \leq qz\}} , \\
            h_2^{(z)}(q, w, v) & := \widetilde{h}^{(z)}(q, w) 1_{\{v \leq (1 - q)z\}}.
        \end{align*}
\end{itemize}
The stochastic processes ($B, N_1, N_2$ and $N_3$) used in the definition of the SDE \eqref{eq:sde_r_general} are assumed 
to be independent and defined in the same complete probability space. 
We now state a result concerning the solution of \eqref{eq:sde_r_general}.

\begin{proposition} \label{prop:important_uniqueness_existence}
    There exists a unique strong solution to \eqref{eq:sde_r_general} with values in 
    $[0, 1]$. Moreover, for every $t > 0$ there exists $K(t) > 0$ such that for any 
    $r, s \in [0, 1]$ we get 
    \begin{equation} \label{eq:lipchitz_initial_condition_r}
        \mathbb{E} \bigl[ \abs{ R_t^{(z, r)} - R_t^{(z, s)} } \bigr] 
        \leq K(t) \abs{r - s} .
    \end{equation}
\end{proposition}

Before presenting the proof, let us make a couple of remarks. 
A key element in the proof of the preceding proposition will be a slight extension of 
Theorem 3.2 in \cite{liStrongSolutionsJumptype2012}, taking advantage that the branching mechanism $\mu_i$ is of bounded variation for cross-branching, meaning that $\int_{U_2} (1 \wedge w_j) \mu_i(dw) < \infty$ whenever $i \neq j$.  
In addition, formally, we should also prove that the solution takes values in $[0, 1]$, 
but this is omitted as it follows along the same lines that the first step in the 
proof of Proposition 4.1 in \cite{caballeroRelativeFrequencyTwo2024}. 
Let us now prove Proposition \ref{prop:important_uniqueness_existence}. 

\begin{proof}
\textit{Pathwise uniqueness.}  
We start by computing some bounds that will be needed in the proof. 
First note that the local Lipschitz assumption on $b_{ij}$ for $i, j \in \{1, 2\}$, 
together with the definition of $\widetilde{D}, S, S_c, m, m_c^1$ and $m_c^2$ in 
\eqref{eq:sdeTerms} respectively, 
implies the existence of a constant $K_d > 0$ such that for any $x, y \in [0, 1]$, 
\begin{align} \MoveEqLeft
    \abs{ \widetilde{D}(x) - \widetilde{D}(y) } 
    + \abs{S(x) - S(y)} + \abs{m(x) - m(y)} \\ 
    & + \abs{S_c(x) - S_c(y)} 
    + \abs{m_c^1(x) - m_c^1(y)} + \abs{m_c^2(x) - m_c^2(y)}
    \leq K_d \abs{x - y}.\label{aux_uni_1}
\end{align}
Moreover, note that for $x, y \in [0, 1]$ we get directly  
\begin{equation}\label{aux_uni_2}
    \int_{U_2} \abs{ \widetilde{g}^{(z)}(x, w) + \widetilde{h}^{(z)}(x, w)
    - \widetilde{g}^{(z)}(y, w) - \widetilde{h}^{(z)}(y, w) } \nu(dw) 
    \leq \abs{x - y} \int_{U_2} \frac{w_1 + w_2}{z + w_1 + w_2} \nu(dw),
\end{equation}
while the bounds 
\begin{align*}
    \int_0^\infty \int_{U_2} \abs{ h_1^{(z)}(x, w, v) - h_1^{(z)}(y, w, v) } \mu_1(dw) dv 
    & \leq 2 \abs{x - y} \int_{U_2} w_2 \mu_1(dw), \\
    \int_0^\infty \int_{U_2} \abs{ g_2^{(z)}(x, w, v) - g_2^{(z)}(y, w, v) } \mu_2(dw) dv 
    & \leq 2 \abs{x - y} \int_{U_2} w_1 \mu_2(dw), 
\end{align*}
follow respectively from the inequalities
\begin{align*}
    \abs{ h_1^{(z)}(x, w, v) - h_1^{(z)}(y, w, v) } 
    & \leq \frac{w_2}{z + w_1 + w_2} ( \abs{x - y} 1_{\{v \leq (x \wedge y) z\}} 
    + 1_{\{ (x \wedge y) z < v \leq (x \vee y) z\}} ), \\
    \abs{ g_2^{(z)}(x, w, v) - g_2^{(z)}(y, w, v) } 
    & \leq \frac{w_1}{z + w_1 + w_2} ( \abs{x - y} 1_{\{v \leq (1 - x \vee y) z\}} 
    + 1_{ \{ (1 - x \vee y) z < v \leq (1 - x \wedge y) z \} } ), 
\end{align*}
valid for any $(w, v) \in U_2 \times \mathbb{R}_+$.  
On the other hand, it is readily seen that for $x, y \in [0, 1]$ and $\sigma$ defined in 
\eqref{eq:sdeTerms}, 
\begin{equation}
    \abs{ \sigma(x) - \sigma(y) }^2 \leq \frac{6 (c_1 + c_2)}{z} \abs{x - y}.\label{aux_uni_3}
\end{equation}
In addition, similar bounds to those obtained for $h_1^{(z)}$ and $g_2^{(z)}$ allow us 
to obtain the inequalities 
\begin{align}
    \int_{U_2} \int_0^\infty \abs{ g_1^{(z)}(x, w, v) - g_1^{(z)}(y, w, v) }^2 dv \mu_1(dw)
    & \leq 2 z \abs{x - y} \int_{U_2} \biggl( \frac{w_1}{z + w_1 + w_2} \biggr)^2 \mu_1(dw), \label{aux_uni_4}
\end{align}
and
\begin{align}
    \int_{U_2} \int_0^\infty \abs{ h_2^{(z)}(x, w, v) - h_2^{(z)}(y, w, v) }^2 dv \mu_1(dw)
    & \leq 2 z \abs{x - y} \int_{U_2} \biggl( \frac{w_2}{z + w_1 + w_2} \biggr)^2 \mu_2(dw),\label{aux_uni_5}
\end{align}
which are valid for any $x , y \in [0, 1]$. 
Finally notice that for any $(w, v) \in U_2 \times \mathbb{R}_2$, the mappings 
\begin{align}
    & x \mapsto x + g_1^{(z)}(x, w, v) = 
    x + (1 - x) \frac{w_1}{z + w_1 + w_2} 1_{\{v \leq xz\}} \\
    \shortintertext{and}
    & x \mapsto x + h_2^{(z)}(x, w, v) = x \biggl( 1 
    - \frac{w_2}{z + w_1 + w_2} 1_{\{v \leq (1 - x) z\}} \biggr)
\end{align}
are non-decreasing for $x \in [0, 1]$. 

Now let $\xi^1$ and $\xi^2$ be two solutions to \eqref{eq:sde_r_general} and set 
$\zeta = \xi^1 - \xi^2$. 
For notational convenience, for $\psi : \mathbb{R}_+ \times U_2 \to \mathbb{R}$ and 
$\rho : \mathbb{R}_+ \times U_2 \times \mathbb{R}_+ \to \mathbb{R}$, we define 
\[
    \Delta_{\psi}(x, y, w) := \psi(x, w) - \psi(y, w) 
    \quad\text{and}\quad
    \Delta_{\rho}(x, y, w, v) := \rho(x, w, v) - \rho(y, w, v)
\]
for $w \in U_2$, $x,y\in\mathbb{R}_+$, and $v \in \mathbb{R}_+$. 
Let us consider the sequence of functions $\{\phi_k : k \in \mathbb{N}\}$ as in 
\cite{liStrongSolutionsJumptype2012}, recalling that they fulfill that 
$\phi_k(z) \to \abs{z}$ non-decreasingly as $k \to \infty$, 
$0 \leq \operatorname{sgn}(z) \phi'(z) \leq 1$, 
and $0 \leq \abs{z} \phi_k''(z) \leq 2 / k$ for every $k \in \mathbb{N}$. 
Then applying Itô's formula yields, for $t \in \mathbb{R}_+$,
\begin{align*}
    \phi_k(\zeta_t) = {} 
    & M_t^k 
    + \int_0^t \phi_k'(\zeta_s) \bigl[ \widetilde{D}(\xi_s^1) - \widetilde{D}(\xi_s^2) \bigr] 
    ds 
    + \int_0^t \phi_k'(\zeta_s) \bigl[ S(\xi_s^1) - S(\xi_s^2) \bigr] ds \\ 
    & + \int_0^t \phi_k'(\zeta_s) \bigl[ S_c(\xi_s^1) - S_c(\xi_s^2) \bigr] ds 
    + \int_0^t \phi_k'(\zeta_s) \bigl[ m(\xi_s^1) - m(\xi_s^2) \bigr] ds \\
    & + \frac{1}{2} \int_0^t \phi_k''(\zeta_s) \bigl[\sigma(\xi_s^1) - \sigma(\xi_s^2)\bigr]^2
    ds \\
    & + \int_0^t \int_{U_2} \Bigl[ 
        \phi_k\bigl( \zeta_s + \Delta_{\widetilde{g}^{(z)} + \widetilde{h}^{(z)}}(
        \xi_s^1, \xi_s^2, w) \bigr) 
    - \phi_k(\zeta_s) \Bigr] \nu(dw) ds \\
    & + \int_0^t \int_{U_2} \int_0^\infty \Bigl[ 
        \phi_k \bigl( \zeta_s + \Delta_{g_1^{(z)} + h_1^{(z)}}(\xi_s^1, \xi_s^2, w, v) \bigr)
    - \phi_k\bigl(\zeta_s + \Delta_{g_1^{(z)}}(\xi_s^1, \xi_s^2, w, v)\bigr) \Bigr] 
    dv \mu_1(dw) ds \\
    & + \int_0^t \int_{U_2} \int_0^\infty \Bigl[ 
        \phi_k \bigl( \zeta_s + \Delta_{g_2^{(z)} + h_2^{(z)}}(\xi_s^1, \xi_s^2, w, v) \bigr)
    - \phi_k\bigl(\zeta_s + \Delta_{h_2^{(z)}}(\xi_s^1, \xi_s^2, w, v)\bigr) \Bigr] 
    dv \mu_2(dw) ds \\
    & + \int_0^t \int_{U_2} \int_0^\infty \Bigl[ 
        \phi_k \bigl( \zeta_s + \Delta_{g_1^{(z)}}(\xi_s^1, \xi_s^2, w, v) \bigr)
    - \phi_k(\zeta_s) - \Delta_{g_1^{(z)}}(\xi_s^1, \xi_s^2, w, v) \phi_k'(\zeta_s) \Bigr] 
    dv \mu_1(dw) ds \\
    & + \int_0^t \int_{U_2} \int_0^\infty \Bigl[ 
        \phi_k \bigl( \zeta_s + \Delta_{h_2^{(z)}}(\xi_s^1, \xi_s^2, w, v) \bigr)
    - \phi_k(\zeta_s) - \Delta_{h_2^{(z)}}(\xi_s^1, \xi_s^2, w, v) \phi_k'(\zeta_s) \Bigr] 
    dv \mu_2(dw) ds, 
\end{align*}
where 
\begin{align*}
    M_t^k = {}
    & \phi_k(\zeta_0) +
    \int_0^t \phi_k'(\zeta_s) \bigl[ \sigma(\xi_s^1) - \sigma(\xi_s^2) \bigr] dB_s \\
    & + \int_0^t \int_{U_2} \Bigl[ 
        \phi_k \bigl( \zeta_{s-} + 
        \Delta_{\widetilde{g}^{(z)} + \widetilde{h}^{(z)}}(
        \xi_{s-}^1, \xi_{s-}^2, w) \bigr)
    - \phi_k(\zeta_{s-}) \Bigr] \widetilde{N}_3(ds, dw) \\
    & + \int_0^t \int_{U_2} \int_0^\infty \Bigl[ 
        \phi_k \bigl( \zeta_{s-} 
        + \Delta_{g_1^{(z)} + h_1^{(z)}}(\xi_{s-}^1, \xi_{s-}^2, w, v) \bigr)
    - \phi_k(\zeta_{s-}) 
    \Bigr] \widetilde{N}_1(ds, dw, dv) \\
    & + \int_0^t \int_{U_2} \int_0^\infty \Bigl[ 
        \phi_k \bigl( \zeta_{s-} 
        + \Delta_{g_2^{(z)} + h_2^{(z)}}(\xi_{s-}^1, \xi_{s-}^2, w, v) \bigr)
    - \phi_k(\zeta_{s-}) 
    \Bigr] \widetilde{N}_2(ds, dw, dv),\qquad t \in \mathbb{R}_+,
\end{align*}
is a martingale due to the fact that $\zeta_s \in [0, 1]$ for all $s \in \mathbb{R}_+$ 
almost surely. 
One might wonder why do the terms involving $\Delta_{h_1^{(z)}}$ and 
$\Delta_{g_2^{(z)}}$ do not appear as compensators in the expansion of $\phi_k(\zeta_t)$, and why neither the terms involving $m_c^1$ and $m_c^2$ appear. 
This is answered by noting that for every $s \in [0, t]$ we have the equalities 
\begin{align*}
    \int_{U_2} \int_0^\infty \Delta_{h_1^{(z)}}(\xi_s^1, \xi_s^2, w, v) dv \mu_1(du)
    & = m_c^1(\xi_s^1) - m_c^1(\xi_s^2) ,\\
    \int_{U_2} \int_0^\infty \Delta_{g_2^{(z)}}(\xi_s^1, \xi_s^2, w, v) dv \mu_2(du)
    & = m_c^2(\xi_s^1) - m_c^2(\xi_s^2),
\end{align*}
so the terms of the form $(m_c^i(\xi_s^1) - m_c^i(\xi_s^2)) \phi_k'(\zeta_s) ds$ end up 
cancelling with the $\Delta$ terms.  
By the previous remarks, we can deduce bounds to obtain the pathwise uniqueness of \eqref{eq:sde_r_general} in a similar fashion to Theorem 3.2 in \cite{liStrongSolutionsJumptype2012}.  
Indeed, we first note that by using \eqref{aux_uni_1} we obtain, for $t \in \mathbb{R}_+$,
\begin{align*} \MoveEqLeft
    \int_0^t \phi_k'(\zeta_s) \Bigl[ 
        \bigl( \widetilde{D}(\xi_s^1) - \widetilde{D}(\xi_s^2) \bigr) 
        + \bigl( S(\xi_s^1) - S(\xi_s^2) \bigr) 
        + \bigl( S_c(\xi_s^1) - S_c(\xi_s^2) \bigr) 
        + \bigl( m(\xi_s^1) -m(\xi_s^2) \bigr)
    \Bigr] ds \\
    & \leq K_d \int_0^t \abs{\zeta_s} \, \abs{\phi_k'(\zeta_s)} ds 
    \leq K_d \int_0^t \abs{\zeta_s} ds,
\end{align*} 
where the second inequality follows from the fact that $\abs{\phi_k'(z)} \leq 1$. 
For the term involving $\sigma$ note that by \eqref{aux_uni_2},
\begin{align}
    \int_0^t \phi_k''(\zeta_s) \bigl[ \sigma(\xi_s^1) - \sigma(\xi_s^2) \bigr]^2 ds 
    & \leq \frac{6 (c_1 + c_2)}{z} \int_0^t \abs{\zeta_s} \phi_k''(\zeta_s) ds
    \leq \frac{12 (c_1 + c_2)}{z k} t,\qquad t \in \mathbb{R}_+,
\end{align}
where for the last inequality we have used that $0 \leq \abs{z} \phi_k''(z) \leq 2/k$. 
To control the term with the measure $\nu$ we note that, using \eqref{aux_uni_3},
\begin{align*} \MoveEqLeft
    \int_0^t \int_{U_2} \bigl[ \phi_k\bigl(\zeta_s + \Delta_{\widetilde{g}^{(z)} +  \widetilde{h}^{(z)}}(\xi_s^1, \xi_s^2, w)\bigr) - \phi_k(\zeta_s) \bigr] \nu(dw) ds \\
    & = \int_0^t \int_{U_2} \int_{\zeta_s}^{\zeta_s + \Delta_{\widetilde{g}^{(z)} +  \widetilde{h}^{(z)}}(\xi_s^1, \xi_s^2, w)} \phi_k'(y) dy \nu(dw) ds \\
    & \leq \int_0^t \int_{U_2} \abs{ \widetilde{g}^{(z)}(\xi_s^1, w) + \widetilde{h}^{(z)}(\xi_s^1, w) - \widetilde{g}^{(z)}(\xi_s^2, w) - \widetilde{h}^{(z)}(\xi_s^2, w) } \nu(dw) ds \\
    & \leq \biggl( \int_{U_2} \frac{w_1 + w_2}{z + w_1 + w_2} \nu(dw) \biggr) \int_0^t \abs{\zeta_s} ds,\qquad t \in \mathbb{R}_+.
\end{align*}
Similarly, we obtain both of the following bounds  
\begin{align*} 
    \MoveEqLeft
    \int_0^t \int_{U_2} \int_0^\infty \Bigl[ 
        \phi_k\bigl( \zeta_s + \Delta_{g_1^{(z)} + h_1^{(z)}}(\xi_s^1, \xi_s^2, w, v) 
        - \phi_k\bigl( \zeta_s + \Delta_{g_1^{(z)}}(\xi_s^1, \xi_s^2, w, v) \bigr)
    \Bigr] dv \mu_1(dw) ds\\
    & \leq \Bigl( 2 \int_{U_2} w_2 \mu_1(dw) \Bigr) \int_0^t \abs{\zeta_s} ds, \\
    \MoveEqLeft
    \int_0^t \int_{U_2} \int_0^\infty \Bigl[ 
        \phi_k\bigl( \zeta_s + \Delta_{g_2^{(z)} + h2^{(z)}}(\xi_s^1, \xi_s^2, w, v) 
        - \phi_k\bigl( \zeta_s + \Delta_{h_2^{(z)}}(\xi_s^1, \xi_s^2, w, v) \bigr)
    \Bigr] dv \mu_2(dw) ds\\
    & \leq \Bigl( 2 \int_{U_2} w_1 \mu_2(dw) \Bigr) \int_0^t \abs{\zeta_s} ds,\qquad t \in \mathbb{R}_+.
\end{align*}
Finally, by using Lemma 3.1 in \cite{liStrongSolutionsJumptype2012}, together with \eqref{aux_uni_4} and \eqref{aux_uni_5}, we get the first inequality  in 
\begin{align*}
    \MoveEqLeft
    \int_0^t \int_{U_2} \int_0^\infty \Bigl[ 
        \phi_k\bigl( \zeta_s + \Delta_{g_1^{(z)}}(\xi_s^1, \xi_s^2, w, v) \bigr)
        - \phi_k(\zeta_s) 
        - \Delta_{g_1^{(z)}}(\xi_s^1, \xi_s^2, w, v) \phi_k'(\zeta_s)
     \Bigr] dv \mu_1(dw) ds \\
     & \leq \int_0^t \int_{U_2} \int_0^\infty \frac{
        2 \bigl(g_1^{(z)}(\xi_s^1, w, v) - g_1^{(z)}(\xi_s^2, w, v)\bigr)^2
     }{k \abs{\zeta_s}} dv \mu_1(dw) ds \\
     & \leq \biggl( \frac{2}{k} \int_{U_2} \biggl(\frac{w_1}{z + w_1 + w_2}\biggr)^2 \mu_1(dw) \biggr) t,\qquad t \in \mathbb{R}_+. 
\end{align*}
Similarly, 
\begin{align*}
    \MoveEqLeft
    \int_0^t \int_{U_2} \int_0^\infty \Bigl[ 
        \phi_k\bigl( \zeta_s + \Delta_{h_2^{(z)}}(\xi_s^1, \xi_s^2, w, v) \bigr)
        - \phi_k(\zeta_s) 
        - \Delta_{h_2^{(z)}}(\xi_s^1, \xi_s^2, w, v) \phi_k'(\zeta_s)
     \Bigr] dv \mu_2(dw) ds \\
     & \leq \biggl(\frac{2}{k} \int_{U_2} \biggl( \frac{w_2}{z + w_1 + w_2} \biggr)^2 \mu_2(dw) \biggr),\qquad t \in \mathbb{R}_+.
\end{align*}
The bounds established give us the inequality, valid for $t \geq 0$,  
\[
    \phi_k(\zeta_t) 
    \leq M_t^k + K' \int_0^t \abs{\zeta_s} ds + \frac{t \tilde{K}}{k},
\]
where 
\begin{align*}
    K'
    & = K_d + 2 \int_{U_2} w_2 \mu_1(dw) + 2 \int_{U_2} w_1 \mu_2(dw) 
    + \int_{U_2} \frac{w_1 + w_2}{z + w_1 + w_2} \nu(dw), \\
    \tilde{K} 
    & = \frac{6 (c_1 + c_2)}{z} 
    + 2 \int_{U_2} \biggl(\frac{w_1}{z + w_1 + w_2}\biggr)^2 \mu_1(dw) 
    + 2 \int_{U_2} \biggl( \frac{w_2}{z + w_1 + w_2} \biggr)^2 \mu_2(dw) .
\end{align*}
By taking expectations, recalling that $M^k$ is a true martingale,
\[
    \mathbb{E}[\phi_k(\zeta_t)] \leq \mathbb{E}[\phi_k(\zeta_0)] 
    + K' \int_0^t \mathbb{E}[\abs{\zeta_s}] ds 
    + \frac{t \tilde{K}}{k} 
\]  
for any $t > 0$. 
Then, letting $k$ tend to infinity and using that $\phi_k(z) \to \abs{z}$ non-decreasingly as $k \to \infty$ we deduce, by the Monotone Convergence Theorem, 
that for any fixed $t > 0$, 
\begin{equation} \label{eq:pre_gronwall}
    \mathbb{E}[\abs{\zeta_t}] \leq \mathbb{E}[\abs{\zeta_0}] 
    + K' \int_0^t \mathbb{E}[\abs{\zeta_s}] ds.
\end{equation}
If $\xi_0^1 = r = \xi_0^2$, then we obtain, by Gronwall's inequality, that 
$\xi_t^1 = \xi_t^2$ almost surely for every $t \geq 0$, from where $\xi^1 = \xi^2$ by the 
right-continuity of the trajectories.  
Moreover, using again Gronwall's inequality, we deduce 
\eqref{eq:lipchitz_initial_condition_r} from \eqref{eq:pre_gronwall} 
by considering $\xi_0^1 = s$ and $\xi_0^2 = r$. 

\textit{Existence of a strong solution.}  
It suffices to note that for $x \in [0, 1]$ we get the bounds  
\begin{align}
    \int_{U_2} \abs{ \widetilde{g}^{(z)}(x, w) + \widetilde{h}^{(z)}(x, w) } \nu(dw) 
    & \leq \int_{U_2} \frac{w_1 + w_2}{z + w_1 + w_2} \nu(dw) \\
    \int_{U_2} \abs{ \widetilde{g}^{(z)}(x, w) + \widetilde{h}^{(z)}(x, w)}^2 \nu(dw) 
    & \leq 2 \int_{U_2} \frac{w_1 + w_2}{z + w_1 + w_2} \nu(dw), \\
    \sum_{i = 1}^2 
    \int_0^\infty \int_{U_2} \abs{g_i^{(z)}(x, w, v) + h_i^{(z)}(x, w, v)}^2 \mu_i(dw) dv 
    & \leq 4 \int_{U_2} \frac{(w_1 + w_2)^2}{z + w_1 + w_2} (\mu_1 + \mu_2)(dw).
\end{align}
Indeed, these bounds along with the definition of $\widetilde{D}$, $S$, $S_c$, $m$, 
$m_c^1$, $m_c^2$ and $\sigma$ imply the existence of a constant $K_l > 0$ such that 
for all $x \in \mathbb{R}$ we have 
\begin{align*}\MoveEqLeft
    \sigma(x)^2 + (\widetilde{D}(x) + S(x) + S_c(x) + m(x) + m_c^1(x) + m_c^2(x))^2 \\
    & + \sum_{i = 1}^2 \int_0^\infty \int_{U_2} 
    \abs{g_i^{(z)}(x, w, v) + h_i^{(z)}(x, w, v)}^2 \mu_i(dw) dv 
    + \int_{U_2} \abs{\widetilde{g}^{(z)}(x, w) + \widetilde{h}^{(z)}(x, w)}^2 \nu(dw) \\
    & + \biggl( \int_{U_2} \abs{ \widetilde{g}^{(z)}(x, w) + \widetilde{h}^{(z)}(x, w) } 
    \nu(dw) \biggr)^2 \leq K_l (1 + x^2) .
\end{align*}
Therefore, the existence of a strong solution, which is \emph{a fortiori} unique, is given by similar arguments to those in the proof of Theorem 5.1 in \cite{liStrongSolutionsJumptype2012}. 
This finishes the proof of Proposition \ref{prop:important_uniqueness_existence}. 
\end{proof}

To use the culling approach we will follow closely the methodology described in 
\cite{caballeroRelativeFrequencyTwo2024}, where the convergence to the limiting 
process $R^{(z, r)}$ is proved using the convergence of infinitesimal generators.  
Hence, we will need $R^{(z, r)}$ to be Feller. 
This is the content of the next result.

\begin{proposition} \label{prop:feller_property}
    For any given $z > 0$, the process $R^{(z, r)}$ is Feller and its infinitesimal 
    generator $\mathcal{L}^{(z)}$ acts on $f \in C^2([0, 1])$ as follows:
    \begin{align}
        \mathcal{L}^{(z)} f(r) = {} 
        & \bigl[ \widetilde{D}(r) + S(r) + m(r) \bigr] f'(r) 
        + \frac{1}{2} \sigma(r)^2 f''(r) \label{eq:ratio_generator} \\
        & + rz \int_{[0, 1]^2} \biggl( f \bigl( r + (1 - r) u_1 - r u_2 \bigr) - f(r) 
        - \frac{(1 - r) u_1}{1 - u_1 - u_2} f'(r) \biggr) \mathbf{T}_2^{(z)} \mu_1(du) \\
        & + (1 - r) z \int_{[0, 1]^2} \biggl( f \bigl( r + (1 - r) u_1 - r u_2 \bigr) 
        - f(r) + \frac{r u_2}{1 - u_1 - u_2} f'(r) \biggr) \mathbf{T}_2^{(z)} \mu_2(du) \\
        & + \int_{[0, 1]^2} \Bigl( f \bigl( r + (1 - r) u_1 - r u_2 \bigr) - f(r) \Bigr) 
        \mathbf{T}_2^{(z)}\nu(du) ,
    \end{align}
    where $\mathbf{T}_2^{(z)} (\cdot)$ is defined in \eqref{eq:t2Def}.
\end{proposition}

\begin{proof}
    We define $T_t f(r) := \mathbb{E}[f(R_t^{(z, r)})]$. The fact that the family $\{T_t : t \geq 0\}$ forms a semigroup such that $T_t(C[0, 1]) \subset C([0, 1])$ 
    for all $t \geq 0$, follows the same lines as the first part of the proof of Proposition 4.1 in \cite{caballeroRelativeFrequencyTwo2024}.

    By Itô's formula we deduce, see Proposition \ref{prop:rz_generator}, that 
    for any $f \in C^2([0, 1])$, 
    \[
        f(R_t^{(z, r)}) = f(r) + \int_0^t \mathcal{L}^{(z)}(R_s^{(z, r)}) ds + M_t^f,
    \]
    where $M^f$ is a local martingale. 
    By using the fact that $f \in C^2([0, 1])$, it is readily seen that 
    \[
        \sup_{r \in [0, 1]} \biggl\lvert \bigl[ \widetilde{D}(r) + S(r) + m(r) \bigr] f'(r) + \frac{1}{2} \sigma(r)^2 f''(r) \biggr\rvert < \infty.
    \]
    On the other hand, for any $r \in [0, 1]$ and $u_1, u_2 \in [0, 1]$ such that $u_1 + u_2 < 1$, standard arguments reveal that 
    \begin{align*}
        \Bigl\lvert f \bigl( r + u_1 (1 - r) - u_2 r \bigr) - f(r) \Bigr\rvert 
        & \leq (u_1 + u_2) \norm{f'}_\infty, \\
        \biggl\lvert f \bigl(r + u_1 (1 - r) - u_2 r \bigr) - f(r) - \frac{(1 - r) u_1}{1 - u_1 - u_2} f'(r) \biggr\rvert 
        & \leq \norm{u}^2 \norm{f''}_\infty + \frac{u_1^2 + u_1 u_2}{1 - u_1 - u_2} \norm{f'}_\infty + u_2 \norm{f'}_\infty, \\
        \biggl\lvert f \bigl(r + u_1 (1 - r) - u_2 r\bigr) - f(r) + \frac{r u_2}{1 - u_1 - u_2} f'(r) \biggr\rvert
        & \leq \norm{u}^2 \norm{f''}_\infty + \frac{u_1 u_2 + u_2^2}{1 - u_1 - u_2} \norm{f'}_\infty + u_1 \norm{f'}_\infty,
    \end{align*}
    from where we deduce 
    the existence of a constant $K > 0$ such that 
    \[
        \sup_{r \in [0, 1]} \abs{ \mathcal{L}^{(z)} f(r) } \leq K.
    \]
    Now the rest of the proof follows as in Proposition 4.1 in \cite{caballeroRelativeFrequencyTwo2024}.
\end{proof}

Just before describing the culling procedure, let us remark that the generator given in 
Proposition \ref{prop:feller_property} coincides with the form of the generator in 
\eqref{eq:rz_generator} when $f$ only depends on $r$ and $z$ is held constant. 
Hence, the process $R^{(z, r)}$ is seen to describe the dynamics of the ratio process $R$ 
under the assumption that the total population stays constant through time. 

\subsection{Culling} 
\label{subsec:culling}
We now turn to show that we can recover the dynamics of $R$ with a Markovian process 
by applying the culling method, which we recall is a sampling technique, 
inspired by the works of Gillespie 
\cite{gillespieNalturalSelectionWithingeneration1974,gillespieNaturalSelectionWithingeneration1975}, 
that takes advantage of the Markovianity of $(R, Z)$ and restarts the process with 
total population $z > 0$ at every sample time. 

Before going into the technical aspect of the process, let us describe the underlying idea behind it. 
For a fixed $n \in \mathbb{N}$ we let the process 
$\{(R_{t \wedge \tau}, Z_{t \wedge \tau}) : t \geq 0\}$ evolve starting from the 
point $(r, z)$, where $\tau$ is defined in Section \ref{subsec:totalPopRatioProcess}. 
At time $t = 1/n$ we sample the first coordinate $R_{n^{-1} \wedge \tau}$, 
which serves as the value to which a continuous-time Markov chain $\overline{R}^{(z, n)}$ jumps to. 
Then, by using that $(R, Z)$ is a homogeneous Markov process, we restart the 
process $\{(R_{t \wedge \tau}, Z_{t \wedge, \tau})\}$ from the point 
$(R_{n^{-1} \wedge \tau}, z)$ by culling (or growing) the population accordingly. 
The procedure just described is then followed inductively. 
Hence, we construct a collection of continuous-time pure jump processes $\{\overline{R}^{(z, n)} : n \in \mathbb{N}\}$ such that for a fixed $n \in \mathbb{N}$, the trajectory $t \mapsto \overline{R}_t^{(z, n)}$ follows the first coordinate of the bivariate process $(R, Z)$, which is observed discretely, and with the property that as $n$ tends to infinity, the fluctuations of $Z$ around $z$ become zero in a dense set of times over $\mathbb{R}_+$. 

Formally, for every $n \in \mathbb{N}$ we define a pure-jump Markov process 
$\overline{R}^{(z, n)}$ with jump times $\{T_m^{(n)} : m \geq 1\}$ given by 
independent exponentially distributed random variables with common rate $n$. 
We denote the law of the process started at $r \in [0, 1]$ by $P_r$. 
The transition kernel $\rho^{(z, n)} : [0, 1] \times \mathcal{B}([0, 1]) \to [0, 1]$ 
of the process will be given by 
\[
    \rho^{(z, n)}(r, A) \equiv P_r \bigl( \overline{R}_{T_1^{(n)}}^{(z, n)} \in A \bigr)
    := \mathbb{P}_{(r, z)} \bigl( R_{n^{-1} \wedge \tau} \in A, 
    Z_{n^{-1} \wedge \tau} \in \mathbb{R}_+ \bigr) ,
\] 
where $\mathbb{P}_{(r, z)}$ stands for the law of $(R, Z)$ started in $(r, z)$. 
With these specifications the infinitesimal generator $\mathcal{A}^{(z, n)}$ of 
$\overline{R}^{(z, n)}$ is bounded, 
see Proposition 17.2 in \cite{kallenbergFoundationsModernProbability2021}, 
and for $f \in C([0, 1])$ we get 
\begin{equation} \label{eq:boundedGen}
    \mathcal{A}^{(z, n)} f(r) = n \int_{[0, 1]} \bigl( f(y) - f(r) \bigr) 
    \rho^{(z, n)}(r, dy).
\end{equation}
Using this process we can obtain an analogous result to Theorem 4.1 of 
\cite{caballeroRelativeFrequencyTwo2024}, which states that the sequence of processes 
$R^{(z, n)}$ converges weakly to the unique solution of \eqref{eq:sde_r_general}.  

\begin{theorem} \label{the:convergence_culling}
    For any fixed $z > 0$ and $T > 0$, $\overline{R}^{(z, n)} \Rightarrow R^{(z, r)}$ 
    as $n \to \infty$ whenever $\overline{R}_0^{(z, n)} \Rightarrow R_0^{(z, r)} = r$ 
    as $n \to \infty$ in 
    $D([0, T], [0, 1])$ with the $J_1$ Skorokhod topology.
\end{theorem}

The proof of the previous result follows by a minor modification of the proof of Theorem 4.1 in \cite{caballeroRelativeFrequencyTwo2024}. 
We include the proof for completeness.

\begin{proof}
    We start by noting that $[0, 1]$ is compact and, by Proposition \ref{prop:feller_property}, $R^{(z, r)}$ is Feller. 
    Therefore, due to Theorem 17.28 in \cite{kallenbergFoundationsModernProbability2021}, it is enough to prove that, 
    for any $f \in C^2([0, 1])$, 
    \[
        \mathcal{A}^{(z, n)} f(r) \to \mathcal{L}^{(z)} f(r) 
        \quad
        \text{uniformly for $r \in [0, 1]$ as $n \to \infty$,}
    \]
    where $\mathcal{L}^{(z)}$ is the generator defined in \eqref{eq:ratio_generator}.
    By definition of $\rho^{(z, n)}$ and \eqref{eq:boundedGen} we obtain 
    \begin{equation} \label{eq:genAsMean}
        \mathcal{A}^{(z, n)} f(r) = n \Bigl( \mathbb{E}_{(r, z)} \bigl[ f(R_{n^{-1} \wedge \tau}) \bigr] - f(r) \Bigr).
    \end{equation}

    Due to Proposition \ref{prop:rz_generator} we know that, for any $f \in C^2([0, 1])$, we have the equality 
    \begin{equation} \label{eq:dynkinApp}
        f(R_{n^{-1} \wedge \tau}) - f(r)
        = \int_0^{n^{-1} \wedge \tau} \mathcal{G} f(R_s, Z_s) ds 
        + M_{n^{-1} \wedge \tau}^f,
    \end{equation}
    where $M^f$ is a local martingale and 
    \begin{align*}
        \mathcal{G} f(r, z) = {}
        & \bigl[ \widetilde{D}(r) + S(r) + m(r) \bigr] f'(r) 
        + \frac{1}{2} \sigma(r)^2 f''(r) \\
        & + rz \int_{U_2} \biggl( 
            f \biggl( r + (1 - r) \frac{w_1}{z + w_1 + w_2} - r \frac{w_2}{z + w_1 + w_2} \biggr)
            - f(r)
            - (1 - r) \frac{w_1}{z} f'(r)
         \biggr) \mu_1(dw) \\
         & + (1 - r) z \int_{U_2} \biggl( 
            f \biggl( r + (1 - r) \frac{w_1}{z + w_1 + w_2} - r \frac{w_2}{z + w_1 + w_2} \biggr)
            - f(r)
            + r \frac{w_2}{z} f'(r)
         \biggr) \mu_2(dw) \\
         & + \int_{U_2} \biggl( 
            f \biggl( r + (1 - r) \frac{w_1}{z + w_1 + w_2} - r \frac{w_2}{z + w_1 + w_2} \biggr)
            - f(r)
         \biggr) \nu(dw) . 
    \end{align*}
    Therefore, as $f \in C^2([0, 1])$ and the fact that, by the definition of $\tau$, $(R_s, Z_s) \in [0, 1] \times [\varepsilon, L]$ for all $s \in [0, n^{-1} \wedge \tau)$, by similar arguments to those given in the proof of Proposition \ref{prop:feller_property}, we deduce the existence of a constant $K_f > 0$ such that 
    \begin{equation} \label{eq:uniformBoundForGenerator}
        \abs{\mathcal{G} f(R_s, Z_s)} \leq K_f \quad \text{for all $s \in [0, n^{-1} \wedge \tau)$} \quad \text{$\mathbb{P}$-a.s.}
    \end{equation}
    This entails that $\{M_{t \wedge \tau}^f : t \in \mathbb{R}_+\}$ is bounded, and thus a martingale that starts in $0$. 

    Hence, taking expectation with respect to $\mathbb{P}_{(r, z)}$ in \eqref{eq:dynkinApp} and using \eqref{eq:genAsMean}, 
    \[
        \mathcal{A}^{(z, n)} f(r) 
        = \mathbb{E}_{r, z} \Bigl[ n \int_0^{n^{-1} \wedge \tau} \mathcal{G} f(R_s, Z_s) ds \Bigr] .
    \]
    Now, by the bound \eqref{eq:uniformBoundForGenerator}, 
    \[
        \Bigl\lvert n \int_0^{n^{-1} \wedge \tau} \mathcal{G} f(R_s, Z_s) ds \Bigr\rvert \leq K_f \quad\text{with whole probability.}
    \]
    Then, we may apply the Dominated Convergence Theorem to deduce  
    \[
        \lim_{n \to \infty} \mathcal{A}^{(z, n)} f(r) 
        = \lim_{n \to \infty} \mathbb{E}_{(r, z)} \Bigl[ n \int_{0}^{n^{-1} \wedge \tau} \mathcal{G} f(R_s, Z_s) ds \Bigr] 
        = \mathcal{G}f(r, z) = \mathcal{L}^{(z)} f(r) .
    \] 
    The uniform convergence follows by Lemma 31.7 of 
    \cite{satoLevyProcessesInfinitely2013}, concluding the proof.
\end{proof}

\section{Moment duality}
\label{sec:duality}

The objective of the section is to present a process that will serve as a moment dual to the culled frequency process defined in section \ref{sec:ratio_process}. 
This will be a block-counting process where blocks are allowed to branch and coalesce, where the branching in the dual process is due to selection and coalescence that arises as a result of the asymmetry in the reproductive mechanisms that we consider in the two-type CBI with competition. 

\subsection{Block-counting process} 
To define the transition rates of the block counting process we will consider, 
let us establish further notation for convenience.  
For $n \in \mathbb{N}$, $k \in \{1, 2, 3, \ldots, n\}$ and $i \in \{1, 2\}$ set 
\begin{align*}
    \lambda_{n, k}^i := {}
    & \int_{[0, 1]^2} u_1^k (1 - u_1 - u_2)^{n-k} \mathbf{T}_2^{(z)} \mu_i(du), \\
    \lambda_{n, k}^I := {} 
    & \int_{[0, 1]^2} u_1^k (1 - u_1 - u_2)^{n-k} \mathbf{T}_2^{(z)} \nu(du). 
\end{align*}
In addition, let us define 
\begin{align*}
    \gamma_n^1 := {} 
    & \int_{[0, 1]^2} \biggl( 1 - (1 - u_1 - u_2)^n - n \frac{u_1}{1 - u_1 - u_2} \biggr) 
    \mathbf{T}_2^{(z)} \mu_1(du), \\
    \gamma_n^2 := {} 
    & \int_{[0, 1]^2} \biggl( 1 - (1 - u_1 - u_2)^n 
    - n \frac{u_2}{1 - u_1 - u_2} \biggr) \mathbf{T}_2^{(z)} \mu_2(du), \\
    \vartheta_{n, \mu} := {}
    & \int_{[0, 1]^2} \bigl( 1 - (1 - u_2)^n \bigr) \mathbf{T}_2^{(z)} \mu_1(du) , \\
    \vartheta_{n, \nu} := {}
    & \int_{[0, 1]^2} \bigl( 1 - (1 - u_2)^n \bigr) \mathbf{T}_2^{(z)} \nu(du) .
\end{align*}
We aim to deduce moment duality, which is why we will consider the generator 
$\mathcal{L}^{(z)}$ in Proposition \ref{prop:feller_property} applied to monomials 
of the form $f(r) = r^n$. 
Therefore we will need to assume that the functions $b_{ij}$ are polynomials, this is 
\begin{equation} \label{eq:assump_poly}
    b_{ij}(x) = \sum_{k = 0}^{m_{ij}} a_{ij}^{(k)} x^k, 
    \quad i, j \in \{1, 2\}. \tag{AP}
\end{equation}

Suppose now that $n \in \mathbb{N}$, $r\in[0,1]$, and $f(r) = r^n$. 
Then we get 
\[
    f'(r) = n r^{n-1} \quad\text{and}\quad f''(r) = 2 \binom{n}{2} r^{n-2},
\]
where we consider $\binom{1}{2} = 0$. 
Thus, by using \eqref{eq:sdeTerms} we deduce 
\[
    \frac{1}{2} \sigma(r)^2 f''(r) 
    = \binom{n}{2} \frac{2}{z} \bigl[ c_1 (r^{n-1} - r^n) + (c_1 - c_2) (r^{n + 1} - r^n) \bigr].
\]
Using $S$ as in \eqref{eq:sdeTerms} we get 
\[
    S(r) f'(r) = n \frac{2}{z} (c_1 - c_2) (r^{n + 1} - r^n),
\]
so 
\begin{align}\MoveEqLeft \label{eq:selection_duality}
    \frac{1}{2} \sigma(r)^2 f''(r) + S(r) f'(r) \\
    & = \binom{n}{2} \frac{2}{z} c_1 (r^{n - 1} - r^n)
    + \binom{n + 1}{2} \frac{2}{z} (c_1 - c_2) (r^{n + 1} - r^n) .
\end{align}
On the other hand, considering $m$ as in \eqref{eq:sdeTerms}, 
\begin{equation} \label{eq:mutation_duality}
    m(r) f'(r) = n \frac{\eta_1}{z} (r^{n-1} - r^n) - n \frac{\eta_2}{z} r^n.
\end{equation}

Let us now consider the term $\widetilde{D}(r) f'(r)$, with $\widetilde{D}$ defined in \eqref{eq:sdeTerms}. 
We start by noting that 
\begin{align*} \MoveEqLeft
    \frac{1 - r}{z} b_{11}(zr) n r^{n-1} \\
    & = n \sum_{k = 0}^{m_{11}} a_{11}^{(k)} z^{k - 1} (r^{n + k - 1} - r^{n + k}) \\
    & = n \frac{a_{11}^{(0)}}{z} (r^{n-1} - r^n) 
    + n \sum_{k = 2}^{m_{11}} a_{11}^{(k)} z^{k - 1} (r^{n + k - 1} - r^n)
    - n \sum_{k = 1}^{m_{11}} a_{11}^{(k)} z^{k - 1} (r^{n + k} - r^n) \\
    & = n \frac{a_{11}^{(0)}}{z} (r^{n-1} - r^n) 
    + n \sum_{k = 1}^{m_{11} - 1} \bigl( a_{11}^{(k + 1)} z^k - a_{11}^{(k)} z^{k - 1} \bigr) (r^{n + k} - r^n) 
    - n a_{11}^{(m_{11})} z^{m_{11} - 1} (r^{n + m_{11}} - r^n).
\end{align*}
Now we note that 
\begin{align*} \MoveEqLeft
    - \frac{r}{z} b_{22}\bigl(z (1 - r)\bigr) n r^{n-1} \\
    & = n \sum_{l = 0}^{m_{22}} (-a_{22}^{(l)} z^{l - 1}) \sum_{k = 0}^l \binom{l}{k} (-1)^k r^{n + k} 
    = n \sum_{k = 0}^{m_{22}} (-1)^k \bigl[ \sum_{l = k}^{m_{22}} \binom{l}{k} (-a_{22}^{(l)} z^{l - 1}) \bigr] r^{n + k} \\
    & = n \sum_{k = 0}^{m_{22}} (-1)^k \bigl[ \sum_{l = k}^{m_{22}} \binom{l}{k} (-a_{22}^{(l)} z^{l - 1}) \bigr] (r^{n + k} - r^n) 
    + n \sum_{k = 0}^{m_{22}} (-1)^k \bigl[ \sum_{l = k}^{m_{22}} \binom{l}{k} (-a_{22}^{(l)} z^{l - 1}) \bigr] r^{n} \\
    & = n \sum_{k = 1}^{m_{22}} (-1)^{k + 1} \bigl[ \sum_{l = k}^{m_{22}} \binom{l}{k} a_{22}^{(l)} z^{l-1} \bigr] (r^{n + k} - r^n) 
    - n r^n \sum_{l = 0}^{m_{22}} a_{22}^{(l)} z^{l - 1} \sum_{k = 0}^{l} \binom{l}{k} (-1)^k \\
    & = n \sum_{k = 1}^{m_{22}} (-1)^{k + 1} \bigl[ \sum_{l = k}^{m_{22}} \binom{l}{k} a_{22}^{(l)} z^{l-1} \bigr] (r^{n + k} - r^n) 
    - n \frac{a_{22}^{(0)}}{z} r^n.
\end{align*}
Similarly, we obtain 
\begin{align*}
    - \frac{r}{z} b_{21}(rz) n r^{n-1} 
    & = \sum_{k = 1}^{m_{21}} n (-a_{21}^{(k)}) z^{k-1} (r^{n + k} - r^n) 
    - n \frac{b_{21}(z)}{z} r^n,
\end{align*}
and 
\begin{align*}
    \frac{1 - r}{z} b_{12} \bigl( z (1 - r) \bigr) n r^{n-1} = {}
    & n \frac{b_{12}(z)}{z} (r^{n-1} - r^n) \\
    & + n \sum_{k = 1}^{m_{12} - 1} (-1)^{k + 1} \bigl[ a_{12}^{(k)} z^{k - 1} + \sum_{l = k + 1}^{m_{12}} \binom{l + 1}{k + 1} a_{12}^{(l)} z^{l - 1} \bigr] (r^{n + k} - r^n) \\
    & + n (-1)^{m_{12} + 1} a_{12}^{(m_{12})} z^{m_{12} - 1} (r^{n + m_{12}} - r^n) .
\end{align*}
Thus, by defining for $k \in \mathbb{N}$,
\begin{align*}
    \theta_k := {} 
    & a_{11}^{(k + 1)} z^k 1_{\{k < m_{11}\}} 
    - a_{11}^{(k)} z^{k - 1} 1_{\{k \leq m_{11}\}} 
    - a_{21}^{(k)} z^{k - 1} 1_{\{k \leq m_{21}\}} 
    + \biggl[ (-1)^{k + 1} \sum_{l = k}^{m_{22}} \binom{l}{k} a_{22}^{(l)} z^{l - 1} \biggr] 1_{\{k \leq m_{22}\}} \\
    & + (-1)^{k + 1} \biggl[ \sum_{l = k + 1}^{m_{12}} \binom{l + 1}{k + 1} a_{12}^{(l)} z^{l - 1} \biggr] 1_{\{k < m_{12}\}} 
    + (-1)^{k + 1} a_{12}^{(k)} z^{k - 1} 1_{\{k \leq m_{12}\}} ,
\end{align*}
and employing the definition of $\widetilde{D}$ in \eqref{eq:sdeTerms} we deduce 
\begin{equation} \label{eq:malthusian_duality}
    \widetilde{D}(r) f'(r) 
    = n \frac{a_{11}^{(0)} + b_{12}(z)}{z} (r^{n - 1} - r^n) 
    + \sum_{k = 1}^{\tilde{m}} n \theta_k (r^{n + k} - r^n) 
    - n \frac{a_{22}^{(0)} + b_{21}(z)}{z} r^n ,
\end{equation}
where $\tilde{m} = m_{11} \vee m_{12} \vee m_{21} \vee m_{22}$.

We now turn to the terms corresponding to jumps in the culled frequency process. 
By using the equalities 
\begin{align*} \MoveEqLeft
    z r \biggl[ \bigl( r + u_1 (1 - r) - u_2 r \bigr)^n - r^n - \frac{(1-r) u_1}{1 - u_1 - u_2} n r^{n - 1} \biggr] \\
    & = z \sum_{k = 1}^n \binom{n}{k} u_1^k (1 - u_1 - u_2)^{n - k} (r^{n - k + 1} - r^n) 
    - z \bigl( 1 - (1 - u_2)^n \bigr) r^n \\
    & \quad - z \biggl[1 - (1 - u_1 - u_2)^n - \frac{n u_1}{1 - u_1 - u_2}\biggr] (r^{n + 1} - r^n), \\ 
    \MoveEqLeft
    z (1 - r) \biggl[ \bigl( r + u_1 (1 - r) - u_2 r \bigr)^n - r^n + \frac{r u_2}{1 - u_1 - u_2} n r^{n - 1} \biggr] \\
    & = z \sum_{k = 2}^{n + 1} \binom{n}{k - 1} u_1^{k - 1} (1 - u_1 - u_2)^{n - k + 1} (r^{n - k + 1} - r^n) \\
    & \quad - z \sum_{k = 2}^n \binom{n}{k} u_1^{k} (1 - u_1 - u_2)^{n - k} (r^{n - k + 1} - r^n) \\
    & \quad + z \biggl[1 - (1 - u_1 - u_2)^n - \frac{n u_2}{1 - u_1 - u_2}\biggr] (r^{n + 1} - r^n) , 
\end{align*}
and
\begin{align*}
    \bigl(r + u_1 (1 - r) - u_2 r\bigr)^n - r^n 
    & = \sum_{k = 1}^n \binom{n}{k} u_1^k (1 - u_1 - u_2)^{n - k} (r^{n - k} - r^n) 
    - \bigl(1 - (1 - u_2)^n\bigr) r^n,
\end{align*}
we deduce 
\begin{align} \label{eq:branching_duality} \MoveEqLeft
    z r \int_{[0, 1]^2} \biggl[f \bigl(r + u_1 (1 - r) - u_2 r\bigr) - f(r) - \frac{(1 - r) u_1}{1 - u_1 - u_2} f'(r)\biggr] \mathbf{T}_2^{(z)} \mu_1(du) \\
    \MoveEqLeft
    \quad + z (1 - r) \int_{[0, 1]^2} \biggl[f \bigl(r + u_1 (1 - r) - u_2 r\bigr) - f(r) + \frac{r u_2}{1 - u_1 - u_2} f'(r)\biggr] \mathbf{T}_2^{(z)} \mu_2(du) \\
    \MoveEqLeft 
    \quad + \int_{[0, 1]^2} \Bigl[ f\bigl( r + u_1 (1 - r) - u_2 r\bigr) - f(r) \Bigr] \mathbf{T}_2^{(z)}\nu(du) \\
    & = \sum_{k = 2}^n \binom{n}{k} z \lambda_{n, k}^1 (r^{n - k + 1} - r^n) 
    - z \vartheta_{n, \mu} r^n - z \gamma_n^1 (r^{n + 1} - r^n) \\
    & \quad + z \lambda_{n, n}^2 (1 - r^n) 
    + \sum_{k = 2}^n \biggl[ \binom{n}{k - 1} z \lambda_{n, k - 1}^2 - \binom{n}{k} z \lambda_{n, k}^2 \biggr] (r^{n - k + 1} - r^n) 
    + z \gamma_{n}^2 (r^{n + 1} - r^n) \\
    & \quad + \sum_{k = 1}^n \binom{n}{k} \lambda_{n, k}^I (r^{n - k} - r^n) 
    - \vartheta_{n, \nu} r^n. 
\end{align}

Finally, let us define the collection $\{q_{nm} : n, m \in \mathbb{N}_0 \cup \{\dagger\}\}$ by setting 
\begin{equation} \label{eq:transition_rates_block_counting}
    q_{nm} = \begin{dcases*}
        \begin{aligned}[b]
            & 1_{\{m = n - 1\}} \biggl( \binom{n}{2} \frac{2}{z} c_1 + n \frac{a_{11}^0 + b_{12}(z) + \eta_1}{z} \biggr) \\
            & \quad + 1_{\{m > 0\}} \binom{n}{n - m + 1} z (\lambda_{n, n - m + 1}^1 - \lambda_{n, n - m + 1}^2) \\
            & \quad + \binom{n}{n - m} z \lambda_{n, n - m}^2 
            + \binom{n}{n - m} \lambda_{n, n - m}^I
        \end{aligned}
        & if $n \in \mathbb{N}$ and $0 \leq m < n$, \\
        1_{\{m = n + 1\}} \biggl( \binom{n + 1}{2} \frac{2}{z} (c_1 - c_2) - z (\gamma_n^1 - \gamma_n^2) \biggr) 
        + n \theta_{m - n} 
        & if $n \in \mathbb{N}$ and $n < m \leq n + \tilde{m}$, \\
        z \vartheta_{n, \mu} + \vartheta_{n, \nu} + n \frac{a_{22}^{(0)} + b_{21}(z) + \eta_2}{z} 
        & if $n \in \mathbb{N}$ and $j = \dagger$, \\
        0 & otherwise.
    \end{dcases*}
\end{equation}
Note that if $q_{nm} \geq 0$ for all $n \neq m$, then this collection defines the 
transition rates of a continuous-time Markov Chain with state space $\mathbb{N}_0 \cup \{\dagger\}$. 
Now we consider the function $H(r, n) : [0, 1] \times (\mathbb{N}_0 \cup \{\dagger\}) \to [0, 1]$ defined by $H(r, n) = r^n$ for $n \in \mathbb{N}_0$ and $H(r, \dagger) = 0$.  
Under the assumption of $q_{nm} \geq 0$ for all $n \neq m$, 
using the equalities \eqref{eq:selection_duality}, \eqref{eq:mutation_duality}, 
\eqref{eq:malthusian_duality} and \eqref{eq:branching_duality} we deduce that for 
any $r \in [0, 1]$ and any $n \in \mathbb{N}_0$, 
\begin{align*}
    \mathcal{L}^{(z)} H(r, n) 
    & = \sum_{m = 0}^{n - 1} q_{n m} (r^m - r^n) 
    + \sum_{m = n + 1}^{n + \tilde{m}} q_{n m} (r^m - r^n) 
    - q_{n \dagger} r^n 
    = \mathcal{Q}^{(z)} H(r, n) ,
\end{align*}
where $\mathcal{Q}^{(z)}$ is the generator defined by the collection 
$\{q_{nm} : n, m \in \mathbb{N}_0 \cup \{\dagger\}\}$. 
This reasoning proves the following theorem.

\begin{theorem}
    Assume that \eqref{eq:assump_poly} holds. 
     If $q_{nm} \geq 0$ for all $n \neq m$, then for any $r \in [0, 1]$ and 
    $n \in \mathbb{N}_0 \cup \{\dagger\}$ and any $t > 0$ we have 
    \[
        \mathbb{E}_r \bigl[ (R_{t}^{(z, r)})^n \bigr]
        = \mathbb{E}_n \bigl[ r^{N_t} \bigr], 
    \]
    where $N$ is a continuous-time Markov Chain over $\mathbb{N}_0 \cup \{\dagger\}$ 
    whose transition rates are defined by \eqref{eq:transition_rates_block_counting}. 
\end{theorem}

\begin{remark}
    As pointed out in section 6 of \cite{caballeroRelativeFrequencyTwo2024}, 
    the assumption $q_{nm} \geq 0$ for all $n \neq m$ restricts the cases we may study 
    by duality to those where there exists a selective advantage of one type (in this case type 1) over the other.  
    Moreover, it imposes restrictions, apart from $b_{12}, b_{21} \geq 0$ on the polynomials we can take into account for the model.  
\end{remark}

\subsection{Biological interpretation of duality}
\label{subsec:biological_duality}
Whenever the dual process of the culled frequency process exists, the rates given in 
\eqref{eq:transition_rates_block_counting} allow us to give a characterization of the evolutionary 
forces that arise due to the different reproduction mechanisms of the two species that constitute the population.  
Most of the mechanisms we will mention also appear in Section 9 of \cite{caballeroRelativeFrequencyTwo2024}, 
but there are new terms that appear in the moment dual because, unlike the population model in \cite{caballeroRelativeFrequencyTwo2024}, we are not assuming independence between the processes $X^{(1)}$ and $X^{(2)}$.
These effects are the following: 
\begin{description}
    \item[Selection] Considering the branching rate in \eqref{eq:transition_rates_block_counting}, we find two types of selection:
    \begin{itemize}
    \item[(i)]\textbf{Classical selection:} This type can be found in the following two terms: 
    The term $\theta_1$, which is the difference between the drift terms; 
    and $2 (c_1 - c_2) / z$, observed previously by Gillespie \cite{gillespieNalturalSelectionWithingeneration1974,gillespieNaturalSelectionWithingeneration1975}, corresponding to the difference between diffusion terms.
    \item[(ii)]\textbf{Frequency dependent selection:} This kind of selection arises in two forms. The first one
    is present in the transition rates \eqref{eq:transition_rates_block_counting} through the terms $\theta_m$ for $m \geq 2$. This type of selection was observed in \cite{gonzalezcasanovaDualityFixationXi2018}, and we want to emphasize that this form of selection can lead to coexistence in the culled frequency process.
    The second one is due to the difference in the jump measures of the two-type CBI with competition in the terms $- z (\gamma_n^1 - \gamma_n^2)$. This kind of selection arises from the asymmetry in large reproduction events between the types in the population and was recently studied in the context of the ancestral selection graph of a $\Lambda$-asymmetric Moran model in \cite{GONZALEZCASANOVA2024}.
    \end{itemize}
    \item[Frequency dependent variance] As noted in \cite{caballeroRelativeFrequencyTwo2024}, from the form of the culled frequency process (given as the solution to \eqref{eq:sde_r_general}), one observes that the asymmetry between the diffusion terms $c_1$ and $c_2$ modifies the variance. In \cite{GONZALEZCASANOVA202033}, this term was obtained in the context of populations
    that require different amounts of resources to reproduce and is referred to as \emph{efficiency}. 
    \item[Coalescence] Both the terms $2 c_1 / z$ and $z (\lambda_{n, n - m + 1}^1 - \lambda_{n, n - m + 1}^2)$ are associated with coalescence; the former being the rate of pairwise mergers, while the latter corresponds to mergers of multiple ancestral lines.  
    \item[Mutation] Mutation is obtained through the inclusion of immigration and cross-branching into the model, and is found in the terms $(a_{11}^0 + b_{12}(z) + \eta_1) / z$, $\lambda_{n, n - m}^2$, $\lambda_{n, n - m}^I$, $(a_{22}^0 + b_{21}(z) + \eta_2) / z$, $z \vartheta_{n, \mu}$ and $\vartheta_{n, \nu}$. Apart from the terms included in \cite{caballeroRelativeFrequencyTwo2024} we obtain new terms thanks to the cross-branching; namely, $\lambda_{n, n -  m}^2$, $b_{12}(z) / z$, $b_{21}(z) / z$ and $z \vartheta_{n, \mu}$.        
\end{description}

\section{Large population limit}
\label{sec:large_population}
In previous sections, we developed the culled frequency process, $R^{(z,r)}$, specifically designed to describe the dynamics of the first coordinate of $(R, Z)$, while keeping the total population size constant at $z>0$. 
In this section, we explore the large population limit of the culled frequency process, proving that it exists under certain hypotheses, and present two examples that exhibit coexistence. 
The latter term means that, in these examples, neither population becomes extinct; equivalently, the frequency process never reaches $0$ or $1$.

\subsection{The large population limit}
\label{subsec:large_population_limit}
We start by noting that if we see the culled frequency process as the solution of 
the SDE \eqref{eq:sde_r_general}, we can actually consider that the functions 
$b_{ij}$ depend on $z$ (which is considered fixed for \eqref{eq:sde_r_general}). 
We stress this fact by writing $b_{ij}^{(z)}$ for the rest of the section. 
This dependence will allow us to introduce suitable rescalings such that the limit 
\begin{equation} \label{eq:limit_bij}
    \lim_{z \to \infty} \frac{b_{ij}^{(z)}(z r)}{z}
\end{equation}
is well defined and finite for all $r \in [0, 1]$. 
To see that this is necessary, notice that if $b_{11}^{(z)}(x) = x (1 - x)$ for all $z > 0$,  
then the limit in \eqref{eq:limit_bij} exists but is equal to $-\infty$, while the limit 
does not exist if $b_{11}^{(z)} = x \sin(x)$ for all $z > 0$. 

Let us assume then that the limit in \eqref{eq:limit_bij} exists uniformly for 
$r \in [0, 1]$ and denote said limit by $\beta_{ij}$. 
Let us further assume that $\beta_{ij}(r) = \lim_{z \to \infty} b_{ij}^{(z)}(rz)/z$ is 
Lipschitz in $[0, 1]$. 
Having stated the assumptions on $\beta_{ij}$, consider the ordinary differential 
equation 
\begin{align}
    d R_t^{(\infty, r)} = {}
    & \beta_{11}(R_t^{(\infty, r)}) (1 - R_t^{(\infty, r)}) dt 
    - \beta_{22}(1 - R_t^{(\infty, r)}) R_t^{(\infty, r)} dt 
    \label{eq:limit_r_ode} \\
    & + \beta_{12}( 1 - R_t^{(\infty, r)} ) ( 1 - R_t^{(\infty, r)} ) dt 
    - \beta_{21}( R_t^{(\infty, r)} ) R_t^{(\infty, r)} dt \\
    & + ( 1 - R_t^{(\infty, r)} )^2 \int_{U_2} w_1 \mu_2(dw)  
    - ( R_t^{(\infty, r)} )^2 \int_{U_2} w_2 \mu_1(dw),
\end{align}
with initial condition $R_0^{(\infty, r)} = r$. 
By comparing \eqref{eq:sde_r_general} with \eqref{eq:limit_r_ode} one might guess that 
the solution of \eqref{eq:limit_r_ode} is the large population limit of the 
culled frequency process. 
Indeed, this is what the following result states.

\begin{theorem} \label{th:large_population_limit}
    For any $T > 0$ we get 
    \[
        \lim_{z \to \infty} \mathbb{E} \biggl[ \sup_{t \leq T} 
        \abs{ R_t^{(z, r)} - R_t^{(\infty, r)} }^2 \biggr] = 0 ;
    \]
    this is, $R^{(z, r)} \to R^{(\infty, r)}$ uniformly in compacts in $L^2$ as 
    $z \to \infty$. 
\end{theorem}

\begin{proof}
    Considering the SDE \eqref{eq:sde_r_general} and the ODE \eqref{eq:limit_r_ode} we get 
    \[
        R_t^{(z, r)} - R_{t}^{(\infty, r)} 
        = A_t^{(z, 1)} + A_t^{(z, 2)} + A_t^{(z, 3)},
    \]
    where 
    \begin{align*}
        A_t^{(z, 1)} = {}
        & \int_0^t \sigma(R_{s-}^{(z, r)}) dB_s 
        + \int_0^t \int_{U_2} \bigl[ \widetilde{g}^{(z)}(R_{s-}^{(z, r)}, w) + 
        \widetilde{h}^{(z)}(R_{s-}^{(z, r)}, w) \bigr] \widetilde{N}_3(ds, dw) \\
        & + \int_0^t \int_{U_2} \int_0^\infty \bigl[ g_1^{(z)}(R_{s-}^{(z,r)}, w, v) + 
        h_1^{(z)}(R_{s-}^{(z,r)}, w, v) \bigr] \widetilde{N}_{1}(ds, dw, dv)  \\
        & + \int_0^t \int_{U_2} \int_0^\infty \bigl[g_2^{(z)}(R_{s-}^{(z,r)}, w, v)
        + h_2^{(z)}(R_{s-}^{(z,r)}, w, v) \bigr] \widetilde{N}_{2}(ds, dw, dv), \\
        A_t^{(z, 2)} = {} 
        & \int_0^t \bigl[ S(R_{s-}^{(z, r)}) + S_c(R_{s-}^{(z, r)}) + m(R_{s-}^{(z, r)}) \bigr] ds
        + \int_0^t \int_{U_2} \bigl[ \widetilde{g}^{(z)}(R_{s-}^{(z, r)}, w)
        + \widetilde{h}^{(z)}(R_{s-}^{(z, r)}, w) \bigr] \nu(dw) ds , \\
        \shortintertext{and}
        A_t^{(z, 3)} = {} 
        & \int_0^t \Bigl[ \widetilde{D}(R_{s-}^{(z, r)}) - 
        \begin{aligned}[t]
            \bigl( & \beta_{11}(R_{s-}^{(\infty, r)}) (1 - R_{s-}^{(\infty, r)}) - \beta_{22}(1 - R_{s-}^{(\infty, r)}) R_{s-}^{(\infty, r)} \\
            & + \beta_{12}(1 - R_{s-}^{(\infty, r)}) (1 - R_{s-}^{(\infty, r)}) - \beta_{21}(R_{s-}^{(\infty, r)}) R_{s-}^{(\infty, r)} \bigr) \Bigr] ds
        \end{aligned} \\
        & + \int_0^t \Bigl[ m_c^1(R_{s-}^{(z, r)}) + (R_{s-}^{(\infty, r)})^2 \int_{U_2} w_2 \mu_1(dw) \Bigr] ds \\
        & + \int_0^t \Bigl[ m_c^2(R_{s-}^{(z, r)}) - (1 - R_{s-}^{(\infty, r)})^2 \int_{U_2} w_1 \mu_2(dw) \Bigr] ds .
    \end{align*}

    We first consider the term $A_t^{(z, 1)}$. 
    Note that by the definition of $\sigma$ given in \eqref{eq:sdeTerms} together with Doob's $L^2$ inequality and Itô's isometry we obtain, for any $t \in [0, T]$,
    \begin{align*}
        \mathbb{E} \biggl[ \sup_{u \leq t} \biggl( \int_0^u \sigma(R_{s-}^{(z, r)}) dB_s \biggr)^2 \biggr] 
        & \leq 4 \mathbb{E} \biggl[ \int_0^T \sigma(R_{s-}^{(z, r)})^2 ds \biggr] 
        \leq \frac{8}{z} (c_1 + c_2) T.
    \end{align*}
    Applying again Doob's $L^2$ inequality we deduce 
    \begin{align*}
        \MoveEqLeft
        \mathbb{E} \biggl[ \sup_{u \leq t} \biggl( \int_0^u \int_{U_2} \bigl[ \widetilde{g}^{(z)}(R_{s-}^{(z, r)}, w) + 
        \widetilde{h}^{(z)}(R_{s-}^{(z, r)}, w) \bigr] \widetilde{N}_3(ds, dw) \biggr)^2 \biggr] \\ 
        & \leq 4 \mathbb{E} \biggl[ \int_0^T \int_{U_2} \bigl[ \widetilde{g}^{(z)}(R_{s-}^{(z, r)}, w) + 
        \widetilde{h}^{(z)}(R_{s-}^{(z, r)}, w) \bigr]^2 \nu(dw) ds \biggr] \\
        & \leq 4 T \int_{U_2} \frac{(w_1 + w_2)^2}{(z + w_1 + w_2)} \nu(dw).
    \end{align*}
    Analogous computations yield 
    \begin{align*}
        \MoveEqLeft
        \sum_{i = 1}^2 \mathbb{E} \biggl[ \sup_{u \leq t} \biggl( \int_0^u \int_{U_2} \int_0^\infty \bigl[ g_{i}^{(z)}(R_{s-}^{(z, r)}, w, v) + 
        h_i^{(z)}(R_{s-}^{(z, r)}, w, v) \bigr] \widetilde{N}_i(ds, dw, dv) \biggr)^2 \biggr] \\
        & \leq 4 T \int_{U_2} \frac{z (w_1 + w_2)^2}{(z + w_1 + w_2)^2} (\mu_1 + \mu_2)(dw) .
    \end{align*}
    By the Dominated Convergence Theorem we note that 
    \[
        \int_{U_2} \frac{(w_1 + w_2)^2}{(z + w_1 + w_2)} \nu(dw)
        + \int_{U_2} \frac{z (w_1 + w_2)^2}{(z + w_1 + w_2)^2} (\mu_1 + \mu_2)(dw) \to 0 \quad\text{as $z \to \infty$.}
    \]
    Hence, we deduce the existence of a collection of constants $\{K_1(T, z) : z > 0\}$ that fulfill $K_1(T, z) \to 0$ as $z \to \infty$ and 
    \[
        \mathbb{E} \Bigl[ \sup_{u \leq t} (A_u^{(z, 1)})^2 \Bigr] 
        \leq K_1(T, z) 
        \quad\text{for all $t \in [0, T]$ and $z > 0$.}
    \]

    Meanwhile, to control the term $A_t^{(z, 2)}$ we simply use the triangle inequality to deduce, for any $t \in [0, T]$, that 
    \begin{align*}
        \mathbb{E} \Bigl[ \sup_{u \leq t} (A_u^{(z, 2)})^2 \Bigr] \leq {}
        & 4 \biggl( \int_0^T \biggl( \frac{\abs{c_2 - c_1}}{2z} + \frac{\eta_1 + \eta_2}{z} \biggr) ds \biggr)^2 
        + 4 \biggl( \int_0^T \int_{U_2} \frac{w_1 (w_1 + w_2)}{z + w_1 + w_2} \mu_1(dw) ds \biggr)^2 \\
        & + 4 \biggl( \int_0^T \int_{U_2} \frac{w_2 (w_1 + w_2)}{z + w_1 + w_2} \mu_2(dw) ds \biggr)^2
        + 4 \biggl( \int_0^T \int_{U_2} \frac{w_1 + w_2}{z + w_1 + w_2} \nu(dw) ds \biggr)^2 \\
        = {} 
        & 4 T^2 \begin{aligned}[t]
            \biggl( & \frac{1}{4 z^2} \bigl( \abs{c_2 - c_1} + 2 \eta_1 + 2 \eta_2 \bigr)^2 
            + \biggl( \int_{U_2} \frac{w_1 (w_1 + w_2)}{z + w_1 + w_2} \mu_1(dw)  \biggr)^2 \\
            & + \biggl( \int_{U_2} \frac{w_2 (w_1 + w_2)}{z + w_1 + w_2} \mu_2(dw)  \biggr)^2  
            + \biggl( \int_{U_2} \frac{w_1 + w_2}{z + w_1 + w_2} \nu(dw) \biggr)^2 \biggr).
        \end{aligned}
    \end{align*}
    Thus, by noting that 
    \[
        \sum_{i = 1}^2 \int_{U_2} \frac{w_i (w_1 + w_2)}{z + w_1 + w_2} \mu_i(dw)
        + \int_{U_2} \frac{w_1 + w_2}{z + w_1 + w_2} \nu(dw)
        \to 0 \quad\text{as } z \to \infty
    \]
    by the Dominated Convergence Theorem, we obtain a collection of constants $\{K_2(T, z) : z > 0\}$ that satisfy both $K_2(T, z) \to 0$ as $z \to \infty$ and 
    \[
        \mathbb{E} \Bigl[ \sup_{u \leq t} (A_u^{(z, 2)})^2 \Bigr] 
        \leq K_2(T, z) 
        \quad\text{for all $t \in [0, T]$ and $z > 0$.}
    \]

    Let us now tackle the term $A_t^{(z, 3)}$. 
    Recall, from \eqref{eq:sdeTerms}, that for $r \in [0, 1]$, 
    \[
        \widetilde{D}(r) 
        = \frac{b_{11}^{(z)}(z r)}{z} (1 - r) 
        - \frac{b_{22}^{(z)}(z (1 - r))}{z} r 
        + \frac{b_{12}{(z)}(z (1 - r))}{z} (1 - r)
        - \frac{b_{21}^{(z)}(z r)}{z} r.
    \]
    Now note that, for $t\in[0,T]$,
    \begin{align*}
        \MoveEqLeft
        \mathbb{E} \biggl[ \sup_{u \leq t} \biggl(
            \int_0^u \biggl(
                \frac{b_{11}^{(z)}(R_{s-}^{(z, r)})}{z} (1 - R_{s-}^{(z, r)})
                - \beta_{11}(R_{s-}^{(\infty, r)}) (1 - R_{s-}^{(\infty, r)})
            \biggr) ds
        \biggr)^2 \biggr] \\
        \leq {} 
        & 3 \mathbb{E} \biggl[ \sup_{u \leq t} \biggl(
            \int_0^u \biggl[
                \frac{b_{11}^{(z)}(z R_{s-}^{(z, r)})}{z}
                - \beta_{11}(R_{s-}^{(z, r)})
            \biggr]
            (1 - R_{s-}^{z, r}) ds
        \biggr)^2 \biggr] \\
        & + 3 \mathbb{E} \biggl[ \sup_{u \leq t} \Bigl(
            \int_0^u \bigl[ \beta_{11}(R_{s-}^{(z, r)}) - \beta_{11}(R_{s-}^{(\infty, r)}) \bigr] (1 - R_{s-}^{(z, r)}) ds
        \Bigr)^2 \biggr] \\
        & + 3 \mathbb{E} \biggl[ \sup_{u \leq t} \Bigl( 
            \int_0^u \beta_{11}(R_{s-}^{(\infty, r)}) (R_{s-}^{(\infty, r)} - R_{s-}^{(z, r)}) ds
         \Bigr)^2 \biggr],
    \end{align*}
    By assumption $\beta_{11}$ is Lipschitz, say with constant $L_{11} > 0$, so 
    $\norm{\beta_{11}}_\infty := \sup \{\abs{\beta_{11}(r)} : r \in [0, 1]\} < \infty$. 
    Moreover, $b_{11}^{(z)}(zr)/z \to \beta_{11}(r)$ uniformly on $r \in [0, 1]$, so there exists a collection of constants $\{C_z^{(11)} : z > 0\}$, with $C_z^{11} \to 0$ as $z \to \infty$, for which 
    \[
        \sup_{r \in [0, 1]} \biggl\lvert \frac{b_{11}^{(z)}(z r)}{z} - \beta_{11}(r) \biggr\rvert \leq C_z^{(11)} \quad\text{for $z > 0$.}
    \]
    By these remarks, we deduce 
    \begin{align*}
        \MoveEqLeft
        \mathbb{E} \biggl[ \sup_{u \leq t} \biggl(
            \int_0^u \biggl(
                \frac{b_{11}^{(z)}(R_{s-}^{(z, r)})}{z} (1 - R_{s-}^{(z, r)})
                - \beta_{11}(R_{s-}^{(\infty, r)}) (1 - R_{s-}^{(\infty, r)})
            \biggr) ds
        \biggr)^2 \biggr] \\
        & \leq 3 T^2 C_z^{(11)} + 3 ( L_{11}^2 + \norm{\beta_{11}}_\infty^2 ) 
        \mathbb{E} \biggl[ \sup_{u \leq t} \Bigl( \int_0^u \abs{R_{s-}^{(z, r)} - R_{s-}^{(\infty, r)}} ds \Bigr)^2 \biggr] \\
        & \leq 3 T^2 C_z^{(11)} + 3 ( L_{11}^2 + \norm{\beta_{11}}_\infty^2 ) T 
        \mathbb{E} \Bigl[ \int_0^t \sup_{u \leq s} \abs{R_{u}^{(z, r)} - R_{u}^{(\infty, r)}}^2 ds \Bigr].
    \end{align*}
    By proceeding in the same way for the rest of the terms in $\widetilde{D}$, it is seen that there is a collection of constants $\{C_z^{(D)} : z > 0\}$, with $C_z^{(D)} \to 0$ as $z  \to \infty$, and a constant $C_D > 0$ such that 
    \begin{align*}
        \mathbb{E} \biggl[ \sup_{u \leq t} \Bigl(
            \int_0^u (\widetilde{D}(R_{s-}^{(z, r)}) - \widetilde{D}_\infty(R_{s-}^{(\infty, r)})) ds 
        \Bigr)^2 \biggr]
        & \leq T^2 C_z^{(D)} + T C_D \mathbb{E} \Bigl[ \int_0^t \sup_{u \leq s} \abs{R_{u}^{(z, r)} - R_{u}^{(\infty, r)}}^2 ds \Bigr],
    \end{align*}
    where 
    \[
        \widetilde{D}_\infty(r) 
        \coloneq \beta_{11}(r) (1 - r) - \beta_{22}(1 - r) r 
        + \beta_{12}(1 - r) (1 - r) - \beta_{21}(r) r . 
    \]
    On the other hand, we see that 
    \begin{align*}
        \MoveEqLeft
        \mathbb{E} \biggl[ \sup_{u \leq t} \biggl( \int_0^u
            \Bigl( m_c^1(R_{s-}^{(z, r)}) + (R_{s-}^{(\infty, r)})^2 \int_{U_2} w_2 \mu_1(dw) \Bigr) ds 
        \biggr)^2 \biggr] \\
        \leq {}
        & 2 \mathbb{E} \biggl[ \sup_{u \leq t} \biggl(
            \int_0^u \int_{U_2} \frac{w_2 (w_1 + w_2)}{z + w_1 + w_2} \mu_1(dw) (R_{s-}^{(z, r)})^2 ds
        \biggr)^2 \biggr] \\
        & + 2 \mathbb{E} \biggl[ \sup_{u \leq t} \Bigl(
            \int_{0}^u \int_{U_2} w_2 \mu_1(dw) \bigl[ (R_{s-}^{(\infty, r)})^2 - (R_{s-}^{(z, r)})^2 \bigr] ds
        \Bigr)^2 \biggr] \\
        \leq {} 
        & 2 \Bigl( \int_{U_2} \frac{w_2 (w_1 + w_2)}{z + w_1 + w_2} \mu_1(dw) \Bigr)^2 T^2 
        + 8 \Bigl( \int_{U_2} w_2 \mu_1(dw) \Bigr)^2 T \mathbb{E} \Bigl[
            \int_0^t \sup_{u \leq s} \abs{R_{u}^{(z, r)} - R_u^{(\infty, r)}}^2 ds
        \Bigr] .
    \end{align*}
    By doing similar computations to the term involving $m_c^2$, from these bounds we deduce the existence of $\{K_3(T, z) : z > 0\}$ and $K_4(T)$ such that $K_3(T, z) \to 0$ as $z \to \infty$ and 
    \[
        \mathbb{E} \Bigl[ \sup_{u \leq t} (A_u^{(z, 3)})^2 \Bigr]
        \leq K_3(T, z) + K_4(T) \mathbb{E} \Bigl[
            \int_0^t \sup_{u \leq s} \abs{R_{u}^{(z, r)} - R_u^{(\infty, r)}}^2 ds
        \Bigr] 
        \quad\text{for all $t \in [0, T]$ and $z > 0$.}
    \]

    As a consequence, for all $t \in [0, T]$ and $z > 0$, we get the 
    inequality 
    \[
        \mathbb{E} \Bigl[ \sup_{u \leq t} \abs{R_{u}^{(z, r)} - R_u^{(\infty, r)}}^2 \Bigr] 
        \leq 3 \bigl(K_1(T, z) + K_2(T, z) + K_3(T, z)\bigr) 
        + 3 K_4(T) \int_0^t \mathbb{E} \Bigl[ \sup_{u \leq s} \abs{ R_{u}^{(z, r)} - R_u^{(\infty, r)} }^2 \Bigr] ds.
    \]
    Now, Gronwall's lemma entails 
    \[
        \mathbb{E} \Bigl[ \sup_{u \leq T} \abs{R_{u}^{(z, r)} - R_u^{(\infty, r)}}^2 \Bigr] 
        \leq 3 \bigl(K_1(T, z) + K_2(T, z) + K_3(T, z)\bigr) \exp\{3 K_4(T) T\}.
    \]
    The proof concludes by noting that 
    $\bigl(K_1(T, z) + K_2(T, z) + K_3(T, z)\bigr) \exp\{3 K_4(T) T\} \to 0$ as 
    $z \to \infty$. 
\end{proof}

Given Theorem \ref{th:large_population_limit}, we now proceed to explain two particular cases where coexistence 
arises in the dynamical system described by \eqref{eq:limit_r_ode} by different 
mechanisms: mutation and logistic growth. 
The analysis done in section \ref{subsec:linear_coeffs} and section \ref{subsec:logistic_growth} is based on the techniques contained in Chapter 2 of \cite{strogatzNonlinearDynamicsChaos2019}, see in particular section 2.2 and section 2.4 of \cite{strogatzNonlinearDynamicsChaos2019}. 

\subsection{Linear coefficients} 
\label{subsec:linear_coeffs}
For the first example we consider $b_{ij}^{(z)}(x) = a_{ij} x$, where $a_{ij}$ is 
constant for $i, j \in \{1, 2\}$ such that $a_{ij} \geq 0$ whenever $i \neq j$. 
Under this assumption \eqref{eq:limit_r_ode} becomes 
\begin{equation} \label{eq:ode_r_linear}
    d R_t^{(\infty, r)} 
    = d_1 R_t^{(\infty, r)} (1 - R_t^{(\infty, r)}) dt + d_2 (1 - R_t^{(\infty, r)})^2 dt 
    - d_3 (R_t^{(\infty, r)})^2 dt, 
\end{equation}
where $d_1, d_2$ and $d_3$ are defined as:
\begin{align*}
    d_1 
    & = a_{11} - a_{22}, \\
    d_2 & = a_{12} + \int_{U_2} w_1 \mu_2(dw) , \\
    d_3 &  = a_{21} + \int_{U_2} w_2 \mu_1(dw) .
\end{align*}
It is not difficult to see that the equilibria that arise in the dynamical system modeled through \eqref{eq:ode_r_linear} are described by:
\begin{enumerate}
    \item When $d_2 - d_1 - d_3 \neq 0$, the equilibria of the system will be given by
        \[
            \frac{2 d_2 - d_1 \pm \sqrt{d_1^2 + 4 d_2 d_3}}{2 (d_2 - d_1 - d_3)},
        \]
        of which 
        \[
            \frac{2 d_2 - d_1 - \sqrt{d_1^2 - 4 d_2 d_3}}{2 (d_2 - d_1 - d_3)},
        \]
        will be the stable one. 
        A complete description of the equilibria is given below: 
        \begin{enumerate}
            \item There is a stable equilibrium at $0$ and an unstable equilibrium at 
                $1$ if $d_2 = d_3 = 0$ and $d_1 < 0$. 
            \item There is an unstable equilibrium at $0$ and a stable equilibrium at 
                $1$ if $d_2 = d_3 = 0$ and $d_1 > 0$.
            \item There is a unique equilibrium, which will be stable, at $0$ if 
                $d_2 = 0$, $d_1 \leq 0$ and $d_3 > 0$. 
            \item There is a unique equilibrium, which will be stable, at $1$ if 
                $d_3 = 0$, $d_1 \geq 0$ and $d_2 > 0$. 
            \item There is an unstable equilibrium at $0$ and a stable equilibrium in 
                $(0, 1)$ if $d_2 = 0$, $d_1 > 0$ and $d_3 > 0$. 
            \item There is an unstable equilibrium at $1$ and a stable equilibrium in 
                $(0, 1)$ if $d_3 = 0$, $d_1 < 0$ and $d_2 > 0$. 
            \item There is a unique equilibrium in $(0, 1)$, which will be stable, 
                if $d_2, d_3 > 0$. 
        \end{enumerate}
    \item When $d_2 - d_1 - d_3 = 0$ there is a unique equilibrium in $[0, 1]$, 
        given by 
        \[
            \frac{d_2}{2 d_2 - d_1},
        \]
        which will be stable. Explicitly, the equilibrium will be:
        \begin{enumerate}
            \item $0$ if $d_2 = 0$ and $d_1 = -d_3$; 
            \item $1$ if $d_1 = d_2$ and $d_3 = 0$; 
            \item not trivial, i.e. it is located in $(0, 1)$, in any other case. 
        \end{enumerate}
\end{enumerate}
Because we are looking for non-trivial equilibria, the cases that we are interested in 
are 1e, 1f, 1g and 2c. 
Figure \ref{fig:linear_coexistence} shows phase diagrams of the behavior of the dynamical system in each of these cases.
Before proceeding to the next example, notice that we can rewrite \eqref{eq:ode_r_linear} 
as 
\[
    d R_t^{(\infty, r)} = (d_1 - d_2 + d_3) R_t^{(\infty, r)} (1 - R_t^{(\infty, r)}) dt 
    + d_2 (1 - R_t^{(\infty, r)}) dt - d_3 R_t^{(\infty, r)} dt .
\]
This form shows that the non-trivial equilibrium that is in the system is due to the 
mutation terms $d_2 (1 - R_t^{(\infty, r)})$ and $d_3 R_t^{(\infty, r)}$, that counterbalance the effect that the ``classical selection'' term $(d_1 - d_2 + d_2) R^{(\infty, r)} (1 - R^{(\infty, r)})$ has. 

\begin{figure}[ht]
    \centering
    \subcaptionbox{Case $d_2 = 0$, $d_1 > 0$ and $d_3 > 0$.}[0.45\linewidth]{%
    \includegraphics[width=\linewidth]{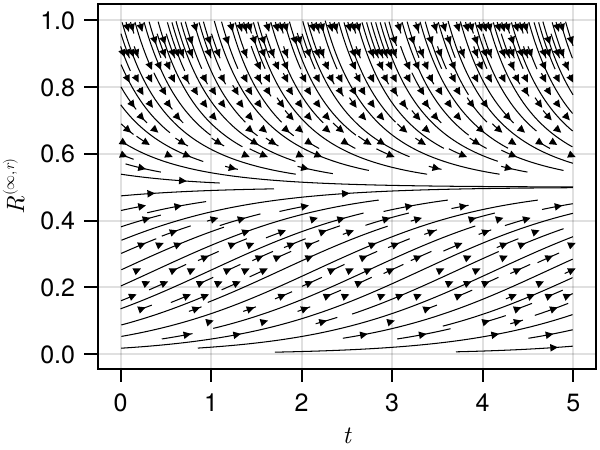}} \hfill %
    \subcaptionbox{Case $d_3 = 0$, $d_1 < 0$ and $d_2 > 0$.}[0.45\linewidth]{%
    \includegraphics[width=\linewidth]{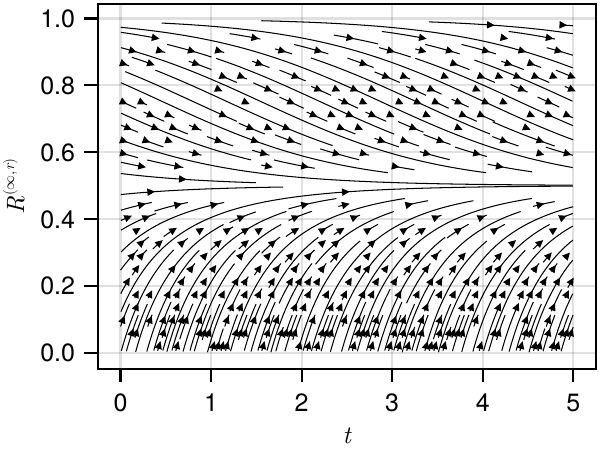}} %
    \subcaptionbox{Case $d_2, d_3 > 0$.}[0.45\linewidth]{%
    \includegraphics[width=\linewidth]{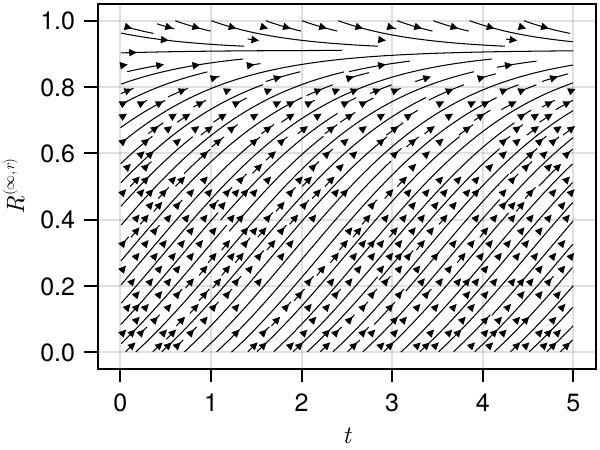}} \hfill
    \subcaptionbox{Case $d_2 - d_1 - d_3 = 0$.}[0.45\linewidth]{%
    \includegraphics[width=\linewidth]{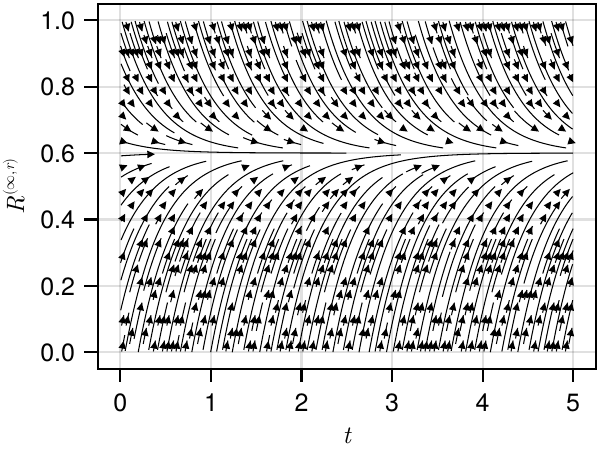}}
    \caption{Examples of non-trivial equilibrium in the linear coefficients case.}
    \label{fig:linear_coexistence}
\end{figure}

\subsection{Logistic growth and no inter-type branching} 
\label{subsec:logistic_growth}
Let us take $b_{ii}^{(z)}(x) = c_{ii} x^2 / z + a_{ii} x$ for $i \in \{1, 2\}$, 
$b_{ij}^{(z)} \equiv 0$ is $i \neq j$, and $\mu_i(\{w_j > 0\}) = 0$ for 
$i \neq j$. Under this assumptions, $\beta_{ii}(x) = c_{ii} x^2 + a_{ii} x$ for 
$i \in \{1, 2\}$ and $\beta_{ij} \equiv 0$, so \eqref{eq:limit_r_ode} can be 
rewritten as 
\begin{equation}
    d R_t^{(\infty, r)} = d_1 R_t^{(\infty, r)} (1 - R_t^{(\infty, r)}) 
    \biggl( R_t^{(\infty, r)} - \frac{d_2}{d_1} \biggr) dt ,
\end{equation}
where 
\begin{align*}
    d_1 & = c_{11} + c_{22}, \\
    d_2 & = c_{22} - a_{11} + a_{22}.
\end{align*}
Notice that the term $R_t^{(\infty, r)} (1 - R_t^{(\infty, r)})$ corresponds to 
classical selection, so if there exists a non-trivial equilibrium in $(0, 1)$, 
it will be due to the \emph{balancing selection} term $R_t^{(\infty, r)} - d_2 / d_1$. 
As we are interested in the non-trivial equilibrium, we will only consider the case where 
$d_1 d_2 > 0$ and $\abs{d_1} > \abs{d_2}$, because only under those conditions there 
will be a non-trivial equilibrium, given by $d_2 / d_1$. 
It can be shown that this equilibrium will be stable if $d_2 < 0$ and unstable if 
$d_2 > 0$. 
Both of these cases are shown graphically, using phase diagrams, in 
figure \ref{fig:logistic_coexistence}.

\begin{figure}[ht]
    \centering
    \subcaptionbox{Case $d_2 > 0$.}[0.45\linewidth]{%
    \includegraphics[width=\linewidth]{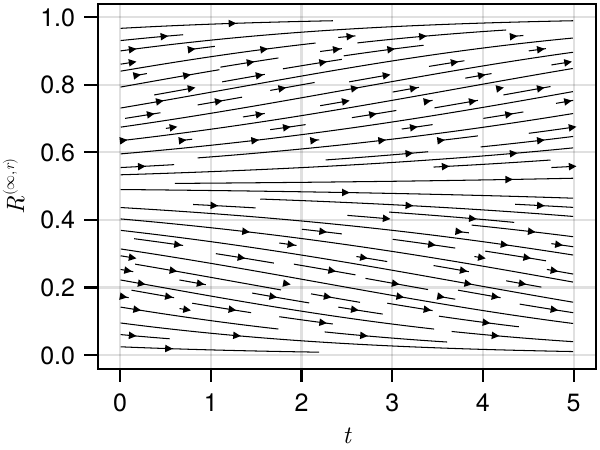}} \hfill%
    \subcaptionbox{Case $d_2 < 0$.}[0.45\linewidth]{%
    \includegraphics[width=\linewidth]{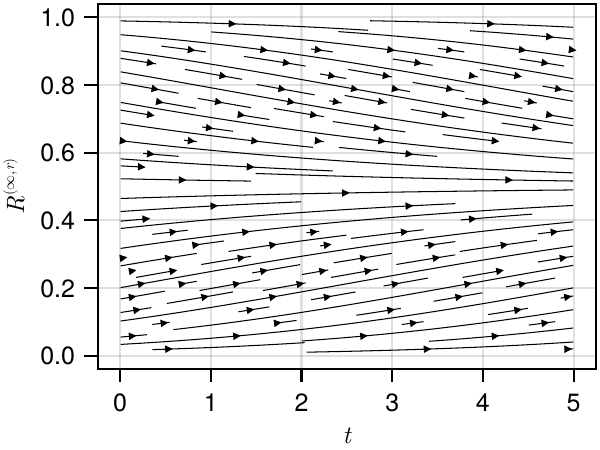}}
    \caption{Examples of non-trivial equilibrium in the logistic growth case.}
    \label{fig:logistic_coexistence}
\end{figure}

What this example shows is that even if there is no multi-type branching, there might exist a stable equilibrium in the large population limit by considering logistic growth for each of the two types within the population. 

Both examples point to the fact that the inclusion of typed branching and general Malthusians, such as the logistic one, may lead to coexistence in the stochastic model. 
We leave this as a future venue for research. 

\nocite{*}
\bibliographystyle{amsplain}
\bibliography{refs.bib}

\end{document}